\documentclass[printtrim,hidelinks,onefignum,onetabnum]{siamart251216}

\usepackage{amsfonts}
\usepackage{amssymb}
\usepackage{mathtools}
\usepackage{algorithmic}
\usepackage{bm}
\usepackage{endnotes}
\usepackage{booktabs}
\usepackage{array}
\usepackage{microtype}
\renewcommand{\thetable}{\arabic{table}}
\renewcommand{\thealgorithm}{\arabic{algorithm}}

\theoremstyle{plain}
\theoremheaderfont{\normalfont\itshape}
\theorembodyfont{\normalfont}
\theoremseparator{.}
\newtheorem{assumption}{Assumption}[section]
\theoremheaderfont{\normalfont\itshape}
\theorembodyfont{\normalfont}
\newtheorem{remark}[theorem]{Remark}
\newtheorem{example}[theorem]{Example}

\DeclareMathOperator*{\argmin}{arg\,min}
\DeclareMathOperator{\dist}{dist}
\DeclareMathOperator{\dom}{dom}
\DeclareMathOperator{\prox}{prox}
\DeclareMathOperator{\grad}{grad}
\DeclareMathOperator{\Retr}{Retr}
\DeclareMathOperator{\crit}{crit}
\DeclareMathOperator{\rank}{rank}
\DeclareMathOperator{\Id}{Id}

\newcommand{\R}{\mathbb{R}}
\newcommand{\N}{\mathbb{N}}
\newcommand{\M}{\mathcal{M}}
\newcommand{\TM}{T\mathcal{M}}
\newcommand{\calF}{\mathcal{F}}
\newcommand{\E}{\mathbb{E}}
\newcommand{\eps}{\epsilon}
\newcommand{\ind}{\iota}
\newcommand{\norm}[1]{\lVert #1\rVert}
\newcommand{\abs}[1]{\lvert #1\rvert}
\newcommand{\ip}[2]{\langle #1,#2\rangle}
\newcommand{\barB}{\overline{B}}
\newcommand{\trans}{\top}
\newcommand{\wh}[1]{\hat{#1}}
\newcommand{\wt}[1]{\widetilde{#1}}

\newcommand{\TheTitle}{An Inexact Riemannian Proximal Momentum Variance-Reduced Method: Complexity Bounds and KL Sequential Convergence}
\newcommand{\TheAuthors}{Na Zhang}
\headers{INEXACT RIEMANNIAN PROXIMAL VARIANCE REDUCTION}{N. ZHANG}
\title{{An Inexact Riemannian Proximal Momentum\\
		Variance-Reduced Method: Complexity Bounds and KL\\
		Sequential Convergence}\thanks{Funding: This work was supported in part by the National Natural Science Foundation of China under grant 12271181 and by the Guangzhou Basic Research Program under grant 2025A04J5240.}}
\author{Na Zhang\thanks{Department of Applied Mathematics, College of
		Mathematics and Informatics, South China Agricultural University, Guangzhou
		510642, China (\email{nazhang2014@scau.edu.cn}).}}

\ifpdf
\hypersetup{
  pdftitle={\TheTitle},
  pdfauthor={\TheAuthors},
  pdfsubject={Riemannian stochastic composite optimization},
  colorlinks=true,
  linkcolor=black,
  citecolor=black,
  urlcolor=black
}
\fi

\begin{document}
\maketitle

\begin{abstract}
We develop a unified analysis of inexact stochastic Riemannian proximal
optimization for finite-sum nonsmooth composite problems over compact embedded
submanifolds. The framework accommodates variance-reduced gradient estimators,
projected momentum, and inexact tangent-space proximal solves under a single
conditional error-dissipation condition, verified for projection-based SVRG,
SARAH/SPIDER, SAGA, and SAG. A computable Fenchel-dual residual criterion,
with tolerance prescribed before sampling and inner iterations, enables explicit
control of the inner work. We establish conditional expected descent,
subsequential stationarity, and an \(O(\epsilon^{-2})\) outer complexity.
With SARAH/SPIDER and accumulative regularization, iRPMVR attains
\(O(n+\sqrt n\,\epsilon^{-2})\) component-gradient and \(O(\epsilon^{-3})\)
proximal-operator complexities. We further develop an abstract KL principle
for conditional expected descent with memory and summable tails using only the
ordinary pointwise KL property. A counterexample shows that a power-type expected-KL implication used in
earlier stochastic analyses can fail. The principle
yields almost-sure finite length, whole-sequence convergence, and deterministic
KL rates.
\end{abstract}

\section{Introduction}
Many problems in statistics, machine learning, signal processing, and scientific
computing involve optimizing a data-fitting objective over a structured manifold
together with a nonsmooth regularizer. Typical examples include sparse principal
component analysis and sparse canonical correlation analysis
\cite{Chen2020Sparse}, sparse subspace clustering \cite{ElhamifarVidal2013}, and
learning models with normalization or orthogonality structures
\cite{Chen2020ManPG,Gao2018,WenYin2013}. In large-scale applications, the smooth
data-fitting term is often formed by averaging a large number of sample losses.
This leads to the finite-sum composite model
\begin{equation}\label{eq:model}
 \min_{x\in\M}\ \Phi(x):=F(x)+R(x),
 \qquad F(x):=\frac1n\sum_{j=1}^n f_j(x),
\end{equation}
where $\M\subseteq\R^d$ is an embedded Riemannian submanifold, each
$f_j:\R^d\to\R$ is smooth on a neighborhood of $\M$, and
$R:\R^d\to\R$ is a proper closed convex function.

This model presents two computational challenges central to our work. First,
when $n$ is large, a full-gradient evaluation requires all $n$ component
gradients and can be prohibitively expensive. Second, the combination of a
nonsmooth regularizer and nonlinear manifold geometry complicates the
construction of practical first-order updates. When $R\equiv0$, the first
challenge has motivated extensive work on Riemannian stochastic-gradient
\cite{Bonnabel2013}, SVRG \cite{ZhangReddiSra2016,Sato2019}, recursive and
SPIDER \cite{Kasai2018,ZhangZhangSra2018}, SAGA \cite{Babanezhad2019},
recursive-momentum \cite{HanGao2021}, and infeasible stochastic
variance-reduced methods for orthogonality constraints \cite{Ablin2024}. By
comparison, stochastic methods with nonasymptotic guarantees remain limited when
$R\not\equiv0$. For the expectation model $F(x)=\E_\xi[f(x,\xi)]$,
\cite{Peng2023} applied stochastic Riemannian gradient steps to a
Moreau-smoothed objective and obtained an $O(\eps^{-5})$ complexity for
generalized stationarity; \cite{Deng2025ALM} developed the double-loop
StoManIAL framework, which uses a Riemannian recursive-momentum method to solve
smooth augmented-Lagrangian subproblems and achieves an
$\widetilde O(\eps^{-3.5})$ stochastic-oracle complexity for a broader
composite model; \cite{Jin2025} proposed the single-loop MARS-ADMM method, which
combines manifold ADMM splitting with a recursive-momentum variance-reduced
estimator and attains $\widetilde O(\eps^{-3})$ iteration and
stochastic-oracle complexity for expected KKT stationarity; and
\cite{Deng2025Smoothing} developed a single-loop stochastic smoothing method
with recursive momentum and obtained an $O(\eps^{-3})$ stochastic-oracle
complexity. Here $\widetilde O(\cdot)$ suppresses logarithmic factors. Since a
finite sum is an expectation under the uniform distribution over a finite
dataset, these guarantees also apply to the corresponding finite-sum
specializations. Along the direct proximal-gradient route, the ManPG paradigm
\cite{Chen2020ManPG,Chen2024Survey} was extended to nonsmooth stochastic
optimization over the Stiefel manifold through R-ProxSGD and R-ProxSPB
\cite{Wang2022}. In the finite-sum setting, R-ProxSPB combines tangent-space
proximal steps with a SpiderBoost estimator and requires
$O(n+\sqrt n\,\eps^{-2})$ incremental first-order oracle (IFO) calls and
$O(\eps^{-2})$ retractions.

The existing approaches therefore exhibit a clear computational tradeoff.
The methods in~\cite{Deng2025ALM,Deng2025Smoothing,Jin2025,Peng2023} avoid the tangent-space proximal subproblems arising in the ManPG paradigm, but their retraction-complexity orders are worse than the \(O(\epsilon^{-2})\) order attained by R-ProxSPB.
This distinction can be computationally significant because a retraction may require, for example, a QR or polar factorization.
On the other hand, R-ProxSPB is restricted to the Stiefel manifold, employs a specific SpiderBoost estimator, and assumes exact tangent-space proximal solves, leaving the computational cost of solving these subproblems outside its complexity analysis.
These observations lead to three related analytical questions.
First, can the convergence and complexity analysis be formulated through a common condition on the gradient-surrogate errors, rather than being tied to a particular variance-reduction mechanism?
Second, if the tangent proximal subproblems are solved only approximately, how should the inexactness be controlled so that the \(O(\epsilon^{-2})\) outer complexity is retained while the total inner computational work can also be quantified?
Third, beyond complexity, can one establish almost-sure whole-sequence convergence under the ordinary pointwise KL property when the available descent relation holds only in conditional expectation?
Addressing these questions requires an analysis that separates the stochastic estimator, the inexact proximal computation, and the KL argument, while allowing their effects to be combined within a common framework.

Motivated by these questions, we develop an analytical framework for inexact stochastic Riemannian proximal optimization over compact embedded submanifolds. The framework is built around a conditional error-dissipation condition for the gradient-surrogate errors and a predictable inexactness rule for the tangent-space proximal subproblems, thereby separating the main convergence and complexity arguments from the particular variance-reduction mechanism and inner solver. Within this framework, we introduce an inexact Riemannian proximal momentum variance-reduced method, termed iRPMVR, which accommodates projected momentum and several representative variance-reduced gradient estimators. Our main contributions are summarized as follows.

First, we introduce a conditional error-dissipation condition that provides a common interface between stochastic gradient estimation and the outer convergence analysis. Under this condition, the descent, stationarity, and complexity arguments are independent of the particular variance-reduction mechanism. We verify the condition for projection-based variants of SVRG, SARAH/SPIDER, SAGA, and SAG, which involve different estimator recursions and memory structures. The condition is also weaker than the multi-part estimator assumption in~\cite{Driggs2021}: its mean-square bound and geometric memory recursion imply our condition after a suitable augmentation of the memory process, whereas the separate first-moment bound and the additional convergence requirement on the gradient estimator are not needed. Combined with the inexactness control described below, the framework yields subsequential stationarity, \(O(\epsilon^{-2})\) outer-iteration and retraction complexity, and corresponding IFO bounds for all four estimators.

Second, we develop a prescribed Fenchel-dual residual criterion for inexact
tangent-space proximal solves that retains the $O(\eps^{-2})$ outer-iteration
and retraction complexity and, for each chosen inner solver, enables an a priori
overall oracle-complexity bound for attaining expected squared criticality at
most $\eps^2$. Related inexact criteria have recently been developed for
deterministic composite proximal-linear and nonsmooth difference-of-convex
models \cite{He2025,Zheng2025,Jiang2025}. Unlike these criteria, our tolerance
is prescribed before sampling the current mini-batch and starting the inner
solve, as the maximum of a history-dependent term and a deterministic floor.
This design is crucial in the stochastic setting. In particular, the relative
criteria in \cite{He2025,Zheng2025} tie the tolerance to the norm of the exact
or current approximate tangent step, which is random, while the outer analysis
controls stationarity only in expectation and provides no pathwise lower bound
on the tolerance. Their deterministic arguments therefore do not directly
bound the total inner work required to meet the expected stationarity criterion;
the deterministic floor in our rule supplies precisely this missing bound. Our
computable residual also controls the tangent-constraint violation, the
projection error, and the deviation from the exact tangent proximal step,
thereby separating the outer analysis from the choice of inner solver. Using
the fast iterative shrinkage-thresholding algorithm (FISTA), Nesterov's
fast-gradient method (NFG), and accumulative regularization (AR) as dual
solvers, we derive corresponding overall complexity bounds in terms of
evaluations of the proximal operator of $R$. The proximal-oracle complexity of
the AR implementation matches the best known order among existing inexact
ManPG-type methods.

Third, we establish an abstract KL principle for conditional expected descent
relations with multiple memory terms and summable tail perturbations. Even in
the deterministic setting, this principle goes beyond classical KL analyses
based on standard sufficient decrease and relative error, since it accommodates
positive delayed terms and error tails while assuming only the ordinary KL
property of the function appearing in the original descent relation. It
requires neither the quasi-additivity condition on the desingularizing function
used in \cite{LiMilzarekQiu2023}, nor a prescribed KL exponent as in
\cite{Qiu2025}, nor an additional KL assumption on an algorithm-dependent
Lyapunov function as in \cite{Ochs2014,PockSabach2016}. The stochastic extension
is more delicate because descent holds in conditional expectation, whereas the
KL inequality is pointwise. We also revisit an expected-KL implication used in stochastic KL analyses.
Specifically, Lemma~4.5 of~\cite{Driggs2021} claims that the ordinary KL property,
together with a common KL exponent, yields an expected-KL inequality relating
the expected objective gap to the expected subdifferential distance through a
single iteration-independent desingularizing function. Its proof, however, is
incorrect: the finite-sum KL calculus invoked there yields, at each iteration
\(k\), only an iteration-dependent desingularizing function
\(\varphi_k(s)=a_k s^{1-\theta}\), while the proof implicitly requires the
coefficients \(\{a_k:k\in\mathbb N\}\) to be uniformly bounded, a property that
does not follow from the stated assumptions. More importantly, this is not
merely a technical gap in the proof. We construct an analytic semialgebraic
counterexample (Example~1) satisfying the standard tail-free conditional
expected descent and relative-error conditions and converging with finite
length, for which no iteration-independent power-type expected-KL inequality
holds for any exponent. We avoid this implication by applying the uniformized ordinary KL inequality
pointwise at the current random iterate before taking conditional expectations,
and then using an augmented supermartingale argument to absorb the memory and
tail terms. Applied to full-step iRPMVR, this principle yields almost-sure
finite length and whole-sequence convergence, provided that the natural
shifted-objective values converge almost surely to a deterministic constant. It
requires the KL property only for this natural shifted objective, rather than
for an auxiliary objective containing additional quadratic terms as in
\cite{He2025,Li2024Structured}. Its deterministic specialization further
provides explicit rates in terms of the KL exponent and the decay of the inner
residuals.
\begin{table}[t]
\caption{Comparison with representative nonsmooth Riemannian methods with
explicit operation complexities for achieving an $\eps$-level first-order
stationarity criterion.}\label{tab:comparison}
\centering
{\scriptsize
\begin{tabular}{@{}lccccc@{}}
\toprule
Algorithm & Stoch. & IFO/SFO & \#Retr & $\#\prox_R$ & KL conv.\\
\midrule
RALM \cite{Deng2025ALM,Xu2025RALM} & No & $O(n\eps^{-3})$ & $O(\eps^{-3})$ & $O(\eps^{-3})$ & No\\
RADMM \cite{Li2025ADMM} & No & $O(n\eps^{-4})$ & $O(\eps^{-4})$ & $O(\eps^{-4})$ & No\\
RSG (Det.)\ \cite{BeckRosset2023,Peng2023} & No & $O(n\eps^{-3})$ & $O(\eps^{-3})$ & $O(\eps^{-3})$ & No\\
RADA \cite{Xu2026RADA} & No & $O(n\eps^{-3})$ & $O(\eps^{-3})$ & $O(\eps^{-3})$ & No\\
OADMM \cite{Yuan2025} & No & $O(n\eps^{-3})$ & $O(\eps^{-3})$ & $O(\eps^{-3})$ & Yes\\
\addlinespace
RSG (Sto.)\ \cite{Peng2023} & Yes & $O(\eps^{-5})$ & $O(\eps^{-5})$ & $O(\eps^{-5})$ & No\\
StoManIAL \cite{Deng2025ALM} & Yes & $\widetilde O(\eps^{-3.5})$ & $\widetilde O(\eps^{-3.5})$ & $\widetilde O(\eps^{-3.5})$ & No\\
MARS-ADMM \cite{Jin2025} & Yes & $\widetilde O(\eps^{-3})$ & $\widetilde O(\eps^{-3})$ & $\widetilde O(\eps^{-3})$ & No\\
Smoothing \cite{Deng2025Smoothing} & Yes & $O(\eps^{-3})$ & $O(\eps^{-3})$ & $O(\eps^{-3})$ & No\\
\addlinespace
iRPDC-BB \cite{Jiang2025} & No & $O(n\eps^{-2})$ & $O(\eps^{-2})$ & $O(\eps^{-4})$ & No\\
iRPDC-NFG \cite{Jiang2025} & No & $O(n\eps^{-2})$ & $O(\eps^{-2})$ & $\widetilde O(\eps^{-3})$ & No\\
iRPDC-AR \cite{Jiang2025} & No & $O(n\eps^{-2})$ & $O(\eps^{-2})$ & $O(\eps^{-3})$ & No\\
IVManPL \cite{He2025} & No & $O(n\eps^{-2})$ & $O(\eps^{-2})$ & $O(\eps^{-3})$ & Yes\\
IManPL \cite{Zheng2025} & No & $O(n\eps^{-2})$ & $O(\eps^{-2})$ & $O(\eps^{-3})$ & No\\
iRPMVR-FISTA (this work) & Yes & $O(n+\sqrt n\eps^{-2})$ & $O(\eps^{-2})$ & $O(\eps^{-4})$ & Yes\\
iRPMVR-NFG (this work)& Yes & $O(n+\sqrt n\eps^{-2})$ & $O(\eps^{-2})$ & $\widetilde O(\eps^{-3})$ & Yes\\
iRPMVR-AR (this work)& Yes & $O(n+\sqrt n\eps^{-2})$ & $O(\eps^{-2})$ & $O(\eps^{-3})$ & Yes\\
\bottomrule
\end{tabular}
}
\begin{minipage}{0.96\textwidth}\scriptsize
``Stoch.'' indicates whether stochastic gradients are used. ``IFO/SFO'' reports
incremental first-order oracle (IFO) calls for finite-sum methods and stochastic
first-order oracle (SFO) calls for expectation-model methods; one full-gradient
evaluation counts as $n$ IFO calls. The symbols \#Retr and $\#\prox_R$ denote
retractions and evaluations of the proximal operator of $R$, respectively, and
``KL conv.'' denotes KL-based whole-sequence convergence. The notation
$\widetilde O(\cdot)$ suppresses logarithmic factors. The IFO entries for
iRPMVR use SARAH/SPIDER.
\end{minipage}
\end{table}

Table~\ref{tab:comparison} provides a benchmark comparison, since the
listed methods address different models and stationarity criteria and are
therefore not directly comparable in every respect. The table reports
separately the costs of stochastic or incremental gradient evaluations,
manifold retractions, and proximal evaluations, which represent the main
computational components of the methods under comparison. With SARAH/SPIDER
and the accumulative regularization (AR) dual solver, iRPMVR requires
\(O(n+\sqrt{n}\epsilon^{-2})\) component-gradient evaluations,
\(O(\epsilon^{-2})\) retractions, and \(O(\epsilon^{-3})\)
proximal-operator evaluations. When \(n\leq \epsilon^{-2}\), its IFO
complexity matches the best known order among stochastic nonsmooth Riemannian
methods. Its retraction complexity and proximal-evaluation complexity match
the best known deterministic orders for every \(n\), with the latter also
matching the best known order among existing inexact ManPG-type methods. In
addition, iRPMVR provides KL-based almost-sure whole-sequence convergence.
Thus, the comparison highlights that the proposed analysis combines
competitive stochastic-gradient complexity with retraction and inner-solve
complexities matching the best known deterministic orders, while retaining
a KL-based sequential convergence guarantee.

The remainder of the paper is organized as follows. Section~2 introduces the notation, basic Riemannian geometry, and the KL property. Section~3 presents iRPMVR and establishes its basic descent properties. Section~4 develops the KL-based sequential convergence theory. Section~5 derives the outer-iteration, IFO, and overall proximal-oracle complexity bounds. Section~6 verifies the abstract estimator condition for projection-based SVRG, SARAH/SPIDER, SAGA, and SAG. Finally, Section~7 concludes the paper.

\section{Notation and preliminaries}
We begin with the notation used throughout the paper. Let $\R^d$ be the ambient
Euclidean space endowed with the inner product $\ip{\cdot}{\cdot}$ and the
induced norm $\norm{\cdot}$. We use $\N:=\{0,1,2,\ldots\}$. For a linear
operator $A$, $A^*$ denotes its adjoint; when $A$ is represented by a matrix,
$A^*=A^\trans$. For $x\in\R^d$ and $\delta>0$, let
$B(x,\delta):=\{y\in\R^d:\norm{y-x}<\delta\}$ and $\barB(x,\delta):=\{y\in\R^d:\norm{y-x}\leq\delta\}.$\\
For a nonempty set $S\subseteq\R^d$, define
$\dist(x,S):=\inf_{y\in S}\norm{x-y}$. If $S$ is nonempty, closed, and convex,
$P_S$ denotes the Euclidean projection onto $S$. The indicator function of $S$
is denoted by $\ind_S$, namely $\ind_S(x)=0$ if $x\in S$, and
$\ind_S(x)=+\infty$ otherwise.

For a proper extended-real-valued function
$\phi:\R^d\to(-\infty,+\infty]$, its domain is
$\dom\phi:=\{x\in\R^d:\phi(x)<+\infty\}$, and $\partial\phi$ denotes its
limiting subdifferential; see \cite[Definition~8.3]{RockafellarWets1998} and
\cite[Section~4.2]{CuiPang2021}. When $\phi$ is proper, closed, and convex,
$\partial\phi$ coincides with the usual convex subdifferential; see, for
example, \cite{Bauschke2017}. The subdifferential domain of $\phi$ is denoted by
$\dom\partial\phi:=\{x\in\dom\phi:\partial\phi(x)\neq\emptyset\}$. For a
proper closed convex function $\phi$, its Fenchel conjugate is defined by
$\phi^*(y):=\sup_{x\in\R^d}\{\ip{x}{y}-\phi(x)\}.$
For $\lambda>0$, the proximal operator of $\phi$ is defined by \cite[Definition~12.23]{Bauschke2017} 
\[\prox_{\lambda\phi}(x):= \argmin_{y\in\R^d}\left\{\phi(y)+\frac{1}{2\lambda}\norm{y-x}^2\right\}.\]

\subsection{Riemannian submanifolds and standing assumptions}
We recall the notation used for embedded submanifolds; see, for example,
\cite{Absil2008,Boumal2023}. A subset $\M\subseteq\R^d$ is called a smooth
embedded submanifold of dimension $d-p$ if, for every $x\in\M$, there exist a
neighborhood $U\subseteq\R^d$ of $x$ and a smooth mapping
$\psi:U\to\R^p$ such that
$\M\cap U=\{y\in U:\psi(y)=0\}$ and $\rank D\psi(x)=p;$
see \cite[Definition~3.10]{Boumal2023}. In this case, $\psi$ is called a local
defining function. The tangent space to $\M$ at $x\in\M$, denoted by $T_x\M$,
satisfies $T_x\M=\ker D\psi(x)$, and the normal space to $\M$ at $x\in\M$ is
$N_x\M=(T_x\M)^\perp$; see \cite[Theorem~3.15]{Boumal2023}. Endowed with the
Riemannian metric induced by the ambient Euclidean inner product, $\M$ is called
a Riemannian submanifold of $\R^d$; see
\cite[Proposition~3.54 and Definition~3.55]{Boumal2023}. Thus, for a smooth
function $h$ defined on a neighborhood of $\M$, its Riemannian gradient
satisfies
$\grad h(x)=P_{T_x\M}\nabla h(x),$
where $\nabla h(x)$ denotes the ambient Euclidean gradient of $h$ at $x$.

Throughout the paper, we work under the following standing assumptions.
\begin{assumption}[Geometry and regularity]\label{ass:geometry}
The following conditions hold.
\begin{enumerate}
\item[(i)] $\M\subseteq\R^d$ is a compact embedded Riemannian submanifold of
dimension $d-p$.
\item[(ii)] Each $f_j$ is continuously differentiable on $\R^d$. There exists
an open neighborhood $U$ of $\M$ such that the ambient gradient of each $f_j$
is Lipschitz continuous on $U$ with a common constant $\ell_f$. Moreover, there
exists a constant $M_f>0$, independent of $n$, such that
$\norm{\nabla f_j(x)}\leq M_f$ for every $x\in\M$ and $j=1,\ldots,n$.
Consequently, $\nabla F$ is Lipschitz continuous with constant $\ell_f$ on
$U$, and $\norm{\grad F(x)}\leq M_f$ for every $x\in\M$.
\item[(iii)] The convex function $R:\R^d\to\R$ is Lipschitz continuous on
$\R^d$ with Lipschitz constant $\ell_R$.
\end{enumerate}
\end{assumption}
The tangent bundle is
$\TM:=\{(x,\eta)\in\R^d\times\R^d:x\in\M,\ \eta\in T_x\M\},$
which is an embedded submanifold of $\R^d\times\R^d$ of dimension
$2\dim\M$; see \cite[Theorem~3.43]{Boumal2023}. A retraction is a smooth
mapping $\Retr:\TM\to\M$ such that $\Retr_x(0)=x$ and
$D\Retr_x(0)=\Id_{T_x\M}$; see \cite[Definition~3.47]{Boumal2023}.

The convergence results will be stated in terms of the following first-order
stationarity notion.

\begin{definition}[Critical point]\label{def:critical}
A point $x\in\M$ is called a critical point of problem~\eqref{eq:model} if and
only if
$0\in\nabla F(x)+\partial R(x)+N_x\M.$
\end{definition}

\subsection{Uniform geometric estimates}
We next collect several uniform estimates that will be used throughout the
analysis. Since $\M$ is compact, there exist finitely many pairs of open sets
$V_\ell,U_\ell\subseteq\R^d$ and smooth maps
$\psi_\ell:U_\ell\to\R^p$, $\ell=1,\ldots,q$, such that
$\overline V_\ell\subseteq U_\ell$, $\{V_\ell\cap\M\}_{\ell=1}^q$ covers
$\M$, and each $\psi_\ell$ is a local defining function of $\M$ on $U_\ell$
with $D\psi_\ell$ of full row rank on $U_\ell\cap\M$. For every $x\in\M$,
choose an index $\ell(x)$ such that $x\in V_{\ell(x)}$ and set
$B_x:=D\psi_{\ell(x)}(x).$
Then $\ker(B_x)=T_x\M$. Since each $\overline V_\ell\cap\M$ is compact, the
singular values of $D\psi_\ell$ on these sets are uniformly bounded above and
away from zero. Consequently, the constants in the estimates below can be
chosen uniformly over $x\in\M$ and independently of the choice of $\ell(x)$.
Their proofs are standard and are omitted for brevity.

\begin{lemma}[Uniform bounds for $B_x$]\label{lem:B-bounds}
Let $\M$ be a compact embedded Riemannian submanifold of $\R^d$. Then the
following statements hold.
\begin{enumerate}
\item[(i)] There exist constants $0<\underline C_B\leq\overline C_B$ such that,
for all $x\in\M$, the singular values of $B_x$ belong to
$[\underline C_B,\overline C_B]$.
\item[(ii)] For all $x\in\M$ and all $v\in\R^d$, there holds
$\norm{B_xv}\leq\overline C_B\norm v$.
\item[(iii)] There exists a constant $C_0>0$ such that, for all $x\in\M$ and
all $v\in\R^d$,
$\dist(v,T_x\M)\leq C_0\norm{B_xv}.$
Equivalently, $\norm{P_{N_x\M}v}\leq C_0\norm{B_xv}$.
\end{enumerate}
\end{lemma}

We next record several uniform retraction estimates that will be used
throughout the analysis; see, e.g., \cite{Chen2020ManPG,HuangWei2022}.

\begin{lemma}[Uniform local estimates for the retraction on a compact manifold]
\label{lem:retraction}
Suppose Assumption~\ref{ass:geometry} holds. Then there exist positive constants
$\delta,C_R,C_Q,L_R,L_\Phi$, and $L_F$ such that, for every $x\in\M$ and every
$\eta\in T_x\M$ with $\norm\eta\leq\delta$, the following estimates hold:
\[
\begin{aligned}
\norm{\Retr_x(\eta)-x}
&\leq C_R\norm{\eta},\\
\norm{\Retr_x(\eta)-x-\eta}
&\leq C_Q\norm{\eta}^2,\\
\left|R(\Retr_x(\eta))-R(x+\eta)\right|
&\leq L_R\norm{\eta}^2,\\
\left|\Phi(\Retr_x(\eta))-\Phi(x+\eta)\right|
&\leq L_\Phi\norm{\eta}^2,\\
F(\Retr_x(\eta))-F(x)-\ip{\grad F(x)}{\eta}
&\leq \frac{L_F}{2}\norm{\eta}^2.
\end{aligned}
\]
\end{lemma}

The next estimate describes how normal vectors to $\M$ can be lifted to normal
vectors to the tangent bundle $\TM$. It is used later to estimate the
subdifferential of the shifted objective function on $\TM$. The proof can be
obtained by adapting the proof of Proposition~23 in \cite{Li2024Structured}.

\begin{lemma}[A normal lifting estimate on $\TM$]\label{lem:normal-lift}
Let $\M$ be a compact embedded Riemannian submanifold of $\R^d$. Then there
exists a constant $C_{\rm tb}>0$ such that, for every $x\in\M$, every
$\eta\in T_x\M$, and every $\zeta\in N_x\M$, there exists $u\in\R^d$
satisfying
\[(u,\zeta)\in N_{\TM}(x,\eta), \quad \norm{u-\zeta}\leq C_{\rm tb}\norm\eta\norm\zeta.\]
In particular, $\norm u\leq(1+C_{\rm tb}\norm\eta)\norm\zeta$.
\end{lemma}

The next estimates follow from the smooth dependence of the tangent-space
projections and Riemannian gradients on the compact manifold $\M$, together
with the retraction estimate in Lemma~\ref{lem:retraction}.

\begin{lemma}[Local and pathwise projection estimates]\label{lem:projection}
Suppose Assumption~\ref{ass:geometry} holds. Then there exist constants
$\delta_T>0$ and $L_T>0$ such that the following statements hold.
\begin{enumerate}
\item[(i)] For every $j\in\{1,\ldots,n\}$, $x\in\M$, and $\eta\in T_x\M$
with $\norm\eta\leq\delta_T$, there holds
\[\norm{\grad f_j(\Retr_x(\eta)) -P_{T_{\Retr_x(\eta)}\M}\grad f_j(x)}\leq L_T\norm\eta.\]
The same estimate holds with $f_j$ replaced by $F$.
\item[(ii)] Let $x^0,\ldots,x^i\in\M$ satisfy
$x^{k+1}=\Retr_{x^k}(\eta^k)$, where $\eta^k\in T_{x^k}\M$ and
$\norm{\eta^k}\leq\delta_T$ for $k=0,\ldots,i-1$. Then, for every
$j\in\{1,\ldots,n\}$, there holds
\[\norm{\grad f_j(x^i)-P_{T_{x^i}\M}\grad f_j(x^0)} \leq L_T\sum_{k=0}^{i-1}\norm{\eta^k}.\]
The same bound holds for $F$.
\end{enumerate}
\end{lemma}

\subsection{KL property}
We recall the Kurdyka--Lojasiewicz (KL) property used in Section~4.

\begin{definition}[KL property and KL exponent \cite{Attouch2013,Bolte2007}]
\label{def:KL}
Let $\phi:\R^d\to(-\infty,+\infty]$ be a proper lower semicontinuous function.
We say that $\phi$ satisfies the KL property at
$\bar x\in\dom\partial\phi$ if there exist $\varepsilon>0$, $\delta>0$, and a
continuous concave function $\varphi:[0,\varepsilon)\to[0,+\infty)$ such that
$\varphi(0)=0$, $\varphi$ is continuously differentiable on
$(0,\varepsilon)$, $\varphi'>0$ on $(0,\varepsilon)$, and
\[\varphi'(\phi(x)-\phi(\bar x))\dist(0,\partial\phi(x))\geq1\]
for all $x\in B(\bar x,\delta)$ satisfying
$\phi(\bar x)<\phi(x)<\phi(\bar x)+\varepsilon$. If, in addition,
$\varphi$ can be chosen as $\varphi(s)=cs^{1-\theta}$ for some $c>0$ and
$\theta\in[0,1)$, then $\phi$ is said to have KL exponent $\theta$ at $\bar x$.
\end{definition}

A proper lower semicontinuous function is called a KL function if it satisfies
the KL property at every point of its subdifferential domain. A useful class of
KL functions is provided by semialgebraic functions. In particular, proper
lower semicontinuous semialgebraic functions are KL functions
\cite{Bolte2007}, and the class of semialgebraic functions is stable under
finite sums, products, quotients with nonvanishing denominators, and addition
of indicator functions of semialgebraic sets \cite{Bolte2014}.

\begin{lemma}[Uniformized KL property \cite{Bolte2014}]\label{lem:uniform-KL}
Let $\Omega\subseteq\R^d$ be a compact set, and let
$\phi:\R^d\to(-\infty,+\infty]$ be a proper lower semicontinuous function
which is constant on $\Omega$. If $\phi$ satisfies the KL property at each point
of $\Omega$, then there exist $\varepsilon>0$, $\delta>0$, and a continuous
concave function $\varphi:[0,\varepsilon)\to[0,+\infty)$ satisfying
$\varphi(0)=0$, $\varphi\in C^1(0,\varepsilon)$, and $\varphi'>0$ on
$(0,\varepsilon)$, such that
\[\varphi'(\phi(z)-\phi(\bar z))\dist(0,\partial\phi(z))\geq1\]
for all $\bar z\in\Omega$ and all $z\in B(\bar z,\delta)$ satisfying
$\phi(\bar z)<\phi(z)<\phi(\bar z)+\varepsilon$.
\end{lemma}
\section{Inexact Riemannian proximal momentum variance reduction}
In this section, we present the proposed inexact Riemannian proximal momentum
variance-reduced framework and derive basic estimates for the inexact tangent
solve.

The algorithm is based on the tangent-space proximal model, which is standard
in ManPG-type methods for nonsmooth manifold optimization
\cite{Chen2020ManPG,Chen2024Survey}. At the $k$-th iteration, let $x^k\in\M$ be
the current iterate and let $m^k\in T_{x^k}\M$ be a tangent surrogate for
$\grad F(x^k)$. Given $m^k$, the exact tangent proximal step is defined by
\begin{equation}\label{eq:exact-step}
 \wh\eta^k=\argmin_{\eta\in T_{x^k}\M}
 \left\{\ip{m^k}{\eta}+\frac{1}{2\alpha}\norm\eta^2+R(x^k+\eta)\right\}.
\end{equation}
In our framework, $m^k$ is generated from variance-reduced gradient information
and projected momentum, and the above tangent proximal subproblem is solved only
approximately. To define the inexact tangent solve, recall that
$\ker(B_x)=T_x\M$. Hence the above subproblem can be written as
\[\min\{\ip{m^k}{\eta}+\frac{1}{2\alpha}\norm\eta^2
+ R(x^k+\eta): B_{x^k}\eta=0\}.\]
Following the Fenchel--Rockafellar framework
\cite[Definition~15.19 and Theorem~15.23]{Bauschke2017}, we introduce
\begin{equation}\label{eq:Fk}
 {F}_k(\eta):=\ip{m^k}{\eta}+\frac{1}{2\alpha}\norm\eta^2
 +R(x^k+\eta),\qquad g(z):=\ind_{\{0\}}(z),
\end{equation}
so that the constrained problem takes the form
$\min_{\eta\in\R^d}{F}_k(\eta)+g(B_{x^k}\eta)$. Since $g^*\equiv0$,
its Fenchel dual is 
\begin{equation}\label{eq:Gk}
\min_{\lambda\in\R^p}\left\{G_k(\lambda):={F}_k^*(-B_{x^k}^{\trans}\lambda)\right\}.
\end{equation}
Since ${F}_k$ is $1/\alpha$-strongly convex, its Fenchel conjugate
${F}_k^*$ has an $\alpha$-Lipschitz continuous gradient
\cite[Corollary~13.33 and Theorem~18.15]{Bauschke2017}. Consequently, $G_k$ is
convex and has an $\alpha\norm{B_{x^k}}^2$-Lipschitz continuous gradient. In the
algorithm, this dual problem is solved only approximately. More precisely, the
inner routine returns a dual vector $\lambda^k\in\R^p$ satisfying the stopping
condition $\norm{\nabla G_k(\lambda^k)}\leq\Delta_k$. Given a dual vector
$\lambda^k\in\R^p$, we associate with it the primal minimizer of the Lagrangian
subproblem
\begin{equation}\label{eq:w-lagrangian}
 w^k=\argmin_{\eta\in\R^d}
 \left\{{F}_k(\eta)+\ip{\lambda^k}{B_{x^k}\eta}\right\}.
\end{equation}
This minimization problem can be written as, by the definitions of
${F}_k$ and proximity operators,
\begin{equation}\label{eq:w-prox}
 w^k:=\prox_{\alpha R}
 \left(x^k-\alpha(m^k+B_{x^k}^{\trans}\lambda^k)\right)-x^k.
\end{equation}
The actual step used by the outer algorithm is the orthogonal projection of
$w^k$ onto the current tangent space:
$\eta^k:=P_{T_{x^k}\M}w^k$, and
$x^{k+1}=\Retr_{x^k}(\gamma_k\eta^k)$ with $0<\gamma_k\le 1$.

With this inexact tangent solve in place, we now describe the stochastic
variance-reduced construction of the gradient estimator. At each iteration
$k$, a random batch $\mathcal{B}_k$ is sampled according to the chosen
variance-reduction module, and this batch is used to construct a tangent
estimator $u^k\in T_{x^k}\M$ of $\grad F(x^k)$. The clipped estimator
$\wh u^k:=P_{\barB(0,M_u)}u^k$ is then combined with the projected momentum by
$m^k=\chi_k\bar m^k+(1-\chi_k)\wh u^k$. The vector $m^k$ is used in the
tangent proximal model, whose subproblem is solved approximately by the dual
residual criterion described above. Throughout, a batch $\mathcal{B}_k$ is
understood as a finite collection of indices drawn from $\{1,\ldots,n\}$.
Under sampling with replacement, repeated indices are allowed, and all sums
over $\mathcal{B}_k$ count multiplicities. This leads to the proposed inexact
Riemannian proximal momentum variance-reduced method, abbreviated as iRPMVR;
see Algorithm~\ref{alg:iRPMVR}. When $\chi_k=0$, $\Delta_k=0$, and the clipping
operation is removed, Algorithm~\ref{alg:iRPMVR} with the SARAH/SPIDER
estimator reduces, on the Stiefel manifold and up to the choice of vector
transport, to the R-ProxSPB method in \cite{Wang2022}.
\begin{center}
\refstepcounter{algorithm}\label{alg:iRPMVR}
{\normalfont\bfseries Algorithm~\thealgorithm.\quad
iRPMVR: An Inexact Riemannian Proximal Momentum Variance-Reduced Method}
\end{center}
\begin{enumerate}
\item Choose $x^0\in\M$, a stepsize $0<\alpha\leq\bar\alpha$, parameter
bounds $0<\gamma\leq1$, $\bar\Delta\geq0$, and $0\leq\bar\chi<1$, and a
clipping radius $M_u:=c_fM_f$, where $c_f\geq1$. The parameters
$\gamma_k,\Delta_k,\chi_k$ are selected before sampling $\mathcal{B}_k$, using
only the available history, and satisfy $0<\gamma\leq\gamma_k\leq1$,
$0\leq\Delta_k\leq\bar\Delta$, and $0\leq\chi_k\leq\bar\chi$. Initialize the
chosen variance-reduction state at $x^0$ so that
$u^0=\wh u^0=\grad F(x^0)$, and set $\bar m^0:=\grad F(x^0)$.
\item For $k=0,1,2,\ldots$, do:
  \begin{enumerate}
  \item[(a)] If $k\neq0$, sample a random batch $\mathcal{B}_k$ from
  $\{1,\ldots,n\}$, with or without replacement, according to the chosen
  variance-reduction module. Construct $u^k\in T_{x^k}\M$, an estimator of
  $\grad F(x^k)$, using $\mathcal{B}_k$ and the current estimator state. Set
  $\wh u^k=P_{\barB(0,M_u)}u^k$.
  \item[(b)] Set $m^k=\chi_k\bar m^k+(1-\chi_k)\wh u^k$.
  \item[(c)] Apply an inner routine to solve \eqref{eq:Gk} to obtain $\lambda^k\in\R^p$ with
  $\norm{\nabla G_k(\lambda^k)}\leq\Delta_k$.
  \item[(d)] Form $w^k$ by \eqref{eq:w-prox} and set
  $\eta^k=P_{T_{x^k}\M}w^k$.
  \item[(e)] Update $x^{k+1}=\Retr_{x^k}(\gamma_k\eta^k)$.
  \item[(f)] Project the momentum by
  $\bar m^{k+1}:=P_{T_{x^{k+1}}\M}m^k$.
  \end{enumerate}
\end{enumerate}

Since the projection onto $\barB(0,M_u)$ is radial, we have
$\wh u^k\in T_{x^k}\M$. Moreover, $\norm{\grad F(x^k)}\leq M_u$, and hence
$\norm{\wh u^k-\grad F(x^k)}\leq\norm{u^k-\grad F(x^k)}$. Finally, the
nonexpansiveness of the projections and an induction imply
\begin{equation}\label{eq:m-bounds}
 \norm{\bar m^k}\leq M_u,\qquad \norm{m^k}\leq M_u,\qquad k\in\N.
\end{equation}

We next state the abstract stochastic assumptions on the gradient-surrogate error. These assumptions will be used throughout the convergence analysis
and will later be verified for the concrete variance-reduction mechanisms
considered in Section~6. Let $\{\calF_k:k\in\N\}$ be the natural filtration
generated by the history of Algorithm~\ref{alg:iRPMVR} before the random
mini-batch $\mathcal{B}_k$ is sampled at iteration $k$. Thus, $x^k$, $\bar m^k$,
and, when needed, the previously constructed estimator $u^{k-1}$ are
$\calF_k$-measurable. The parameters $\gamma_k,\Delta_k,\chi_k$ are also
$\calF_k$-measurable since they are selected before sampling $\mathcal{B}_k$.
The quantities $\wh u^k,m^k,\lambda^k,w^k,\eta^k$ are then generated after
sampling $\mathcal{B}_k$. We write
$\E_k[\cdot]:=\E[\cdot\mid\calF_k]$. The following conditions are imposed for
every iteration $k$ and every admissible history generated by the algorithm.

\begin{assumption}[Abstract condition on the gradient-surrogate error]
\label{ass:estimator}
Let $e^k:=m^k-\grad F(x^k)$. There exist constants $c_e>0$, $\mu>0$,
$\nu>0$, and a sequence of nonnegative random variables
$\{\mathcal E_k:k\in\N\}$, with $\mathcal E_k$ being $\calF_k$-measurable for each
$k$ and $\E[\mathcal E_0]<+\infty$, such that for every $k\geq0$,
\begin{equation*}
 \E_k[\norm{e^k}^2]+\mu\mathcal E_k
 \leq c_e\bigl(\mathcal E_k-\E_k[\mathcal E_{k+1}]\bigr)
 +\nu\E_k[\norm{\eta^k}^2].
\end{equation*}
\end{assumption}

Assumption~\ref{ass:estimator} is closely related to the variance-reduced
estimator condition in \cite[Definition~2.1]{Driggs2021}. The mean-square bound
and geometric recursion in that condition can be combined, after a suitable
augmentation of the memory process, to yield an inequality of the above form.
In contrast, our analysis requires neither the separate first-moment bound nor
the additional estimator-convergence condition imposed there.

\subsection{Basic estimates for the inexact tangent solve}
We first relate the dual residual to the resulting primal inexactness.

\begin{lemma}[Primal error bounds from dual residual]\label{lem:primal-error}
	Suppose Assumption~\ref{ass:geometry} holds. Then there exist constants
	$C_1,C_2>0$, independent of $k$ and $n$, such that
	\[
	\begin{aligned}
		\norm{B_{x^k}w^k}&\leq\Delta_k, \qquad
		\norm{w^k-\eta^k}\leq C_0\Delta_k,\\
		\norm{\eta^k-\wh\eta^k}&\leq C_1\sqrt{\Delta_k}, \qquad
		\norm{w^k-\wh\eta^k}\leq C_2\sqrt{\Delta_k}
	\end{aligned}
	\]
	hold for all $k\geq0$, where $C_0$ is from
	Lemma~\ref{lem:B-bounds}.
\end{lemma}
\begin{proof}
Fix $k\in\N$ and abbreviate \[x:=x^k, \quad B:=B_x, \quad w:=w^k, \quad \eta:=\eta^k,\quad \wh\eta:=\wh\eta^k.\]
Since
$G_k(\lambda)= F_k^*(-B_{x^k}^{\trans}\lambda)$, Danskin's theorem
\cite[Proposition~B.22]{Bertsekas2016} and the definition of $F_k$
in~\eqref{eq:Fk}
yield $\nabla G_k(\lambda^k)=-B_{x^k}w^k$. Hence the stopping rule implies
$\norm{B_{x^k}w^k}\leq\Delta_k$. Since
$\eta=P_{T_x\M}w$ and $\ker(B)=T_x\M$, Lemma~\ref{lem:B-bounds} implies
\[\norm{w-\eta}=\dist(w,T_x\M)\leq C_0\norm{Bw}\leq C_0\Delta_k.\]

We next estimate $\norm{\eta-\wh\eta}$.
Write $w=\eta+\vartheta$ with $\eta\in T_x\M$ and $\vartheta\in N_x\M$.
Then $\norm\vartheta\leq C_0\Delta_k$. By optimality of $w$ in the
Lagrangian subproblem~\eqref{eq:w-lagrangian}, there exists
$z\in\partial R(x+w)$ such that
$m^k+B^\trans\lambda^k+\frac1\alpha w+z=0.$
Hence,
\[m^k+\frac1\alpha\eta+z=-B^\trans\lambda^k-\frac1\alpha\vartheta.\]
Let $v\in T_x\M$. Since $\wh\eta$ minimizes $F_k$ over $T_x\M$,
$F_k(\wh\eta)\leq F_k(v)$. By convexity of $R$ and the
choice of $z$, we have $R(x+v)\geq R(x+w)+\ip{z}{v-w}$. Using this in
$F_k(v)-F_k(\eta)$ leads to
\[F_k(v)-F_k(\eta)\geq
\ip{m^k}{v-\eta}+\frac1{2\alpha}(\norm v^2-\norm\eta^2)
+\ip{z}{v-w}+R(x+w)-R(x+\eta).\]
Since $w=\eta+\vartheta$, this becomes
\[F_k(v)-F_k(\eta)\geq
\ip{m^k+\frac1\alpha\eta+z}{v-\eta}
+\frac1{2\alpha}\norm{v-\eta}^2
-\ip{z}{\vartheta}+R(x+w)-R(x+\eta).\]
Substituting
$m^k+\frac1\alpha\eta+z=-B^\trans\lambda^k-\frac1\alpha\vartheta$ and using
$v,\eta\in T_x\M=\ker(B)$ together with the fact that
$\vartheta\in N_x\M$, we see that the first inner product term vanishes.
Since $R$ is convex and globally $\ell_R$-Lipschitz by
Assumption~\ref{ass:geometry}, we get $\norm z\leq\ell_R$.
Also, $\abs{R(x+w)-R(x+\eta)}\leq\ell_R\norm{w-\eta}
=\ell_R\norm\vartheta$. Thus,
\[-\ip{z}{\vartheta}+R(x+w)-R(x+\eta)\geq-2\ell_R\norm\vartheta.\]
Therefore, for any $v\in T_x\M$,
\begin{equation}\label{eq:basic-Fk}
 F_k(v)-F_k(\eta^k)
 \geq \frac{1}{2\alpha}\norm{v-\eta^k}^2-2\ell_R\norm\vartheta
 \geq \frac{1}{2\alpha}\norm{v-\eta^k}^2-2\ell_RC_0\Delta_k.
\end{equation}
Taking $v=\wh\eta$ and using the optimality of $\wh\eta$, namely
$F_k(\wh\eta)\leq F_k(\eta)$, we get from
$\norm\vartheta\leq C_0\Delta_k$ that
$\norm{\eta-\wh\eta}\leq C_1\sqrt{\Delta_k}$ with
$C_1:=2\sqrt{\ell_RC_0\bar\alpha}$.

Finally, by the triangle inequality,
$\norm{w-\wh\eta}\leq\norm{w-\eta}+\norm{\eta-\wh\eta}
\leq C_0\Delta_k+C_1\sqrt{\Delta_k}$. Since $\Delta_k$ is uniformly bounded,
the right-hand side is bounded by $C_2\sqrt{\Delta_k}$ for some constant
$C_2>0$. This completes the proof.
\end{proof}

The preceding residual bounds also yield uniform step-size estimates. In
particular, for sufficiently small $\alpha$, all inexact tangent steps remain
in the neighborhood where the retraction and projection estimates of Section~2
are valid.

\begin{lemma}\label{lem:step-size}
Suppose Assumption~\ref{ass:geometry} holds. Then, for all $k\geq0$, we have
\[\norm{\eta^k}\leq(M_u+\ell_R)\alpha
+2\sqrt{\ell_RC_0\bar\Delta}\sqrt\alpha \quad \text{and} \quad 
\norm{\wh\eta^k}\leq(M_u+\ell_R)\alpha.\]
Moreover, if \[0<\bar\alpha\leq C_\alpha:= \left( \frac{\sqrt{\ell_RC_0\bar\Delta+(M_u+\ell_R)\min\{\delta,\delta_T\}} -\sqrt{\ell_RC_0\bar\Delta}}{M_u+\ell_R} \right)^2,\]
then we have
\[\max\{\norm{\eta^k},\norm{\wh\eta^k}\}\leq\min\{\delta,\delta_T\} \quad \text{for all} \quad k\geq0,\] where $\delta$ and $\delta_T$ are from
Lemmas~\ref{lem:retraction} and \ref{lem:projection}, respectively.
\end{lemma}
\begin{proof}
Fix $k\in\N$ and write $x:=x^k$ and $\eta:=\eta^k$. Taking $v=0$ in
\eqref{eq:basic-Fk} yields
\[F_k(\eta)\leq F_k(0)-(2\alpha)^{-1}\norm\eta^2
+2\ell_RC_0\Delta_k.\] Expanding $F_k$ and using
$\norm{m^k}\leq M_u$ and
$R(x+\eta)\geq R(x)-\ell_R\norm\eta$, we obtain
\[\alpha^{-1}\norm\eta^2 \leq(M_u+\ell_R)\norm\eta+2\ell_RC_0\Delta_k.\]
Since $\Delta_k\leq\bar\Delta$ and $\alpha\leq\bar\alpha$, Young's inequality
yields
\[\norm\eta^2\leq\frac12\norm\eta^2
+\frac{\alpha^2}{2}(M_u+\ell_R)^2
+2\alpha\ell_RC_0\bar\Delta.\] Thus the first inequality holds.

For $\wh\eta^k$, the strong convexity of $F_k$ on $T_{x^k}\M$
implies
\[ F_k(\wh\eta^k)\leq F_k(0)-(2\alpha)^{-1}
\norm{\wh\eta^k}^2.\] On the other hand, the globally Lipschitz continuity of
$R$, \eqref{eq:m-bounds} and \eqref{eq:Fk} yield
\[ F_k(\wh\eta^k)\geq F_k(0)
-(M_u+\ell_R)\norm{\wh\eta^k}
+(2\alpha)^{-1}\norm{\wh\eta^k}^2.\]
Combining the two inequalities implies the second inequality. The last
assertion follows directly from the choice of $C_\alpha$.
\end{proof}

\subsection{Function descent estimates}
We next derive a conditional expected function descent estimate by combining
the inexact tangent-solve bounds with the retraction estimates and the abstract
stochastic estimator condition.

\begin{theorem}[Conditional expected function descent]\label{thm:descent}
Suppose Assumptions~\ref{ass:geometry} and \ref{ass:estimator} hold. Assume
that
$0<\bar\alpha\leq\min\{C_\alpha,1/(L_F+2L_R)\},$
where $C_\alpha$ is from Lemma~\ref{lem:step-size}. Let
$\{(x^k,\eta^k):k\in\N\}$ be generated by Algorithm~\ref{alg:iRPMVR}, and let
$\wh\eta^k$ be defined by \eqref{eq:exact-step}. Define
$\sigma_0:=1/(2\alpha)-L_F/2-L_R$. Then, for every $k\geq0$ and every
$\tau>0$,
\begin{equation*}
\begin{aligned}
 \E_k[\Phi(x^{k+1})]
 \leq{}&\Phi(x^k)-\left(\sigma_0-\frac\tau2\right)\gamma_k
 \E_k[\norm{\eta^k}^2]
 -\frac{\gamma_k}{2\alpha}\E_k[\norm{\wh\eta^k}^2]\notag\\
 &+\frac{c_e}{2\tau}\bigl(\mathcal E_k-
 \E_k[\mathcal E_{k+1}]\bigr)-\frac{\mu}{2\tau}\mathcal E_k
 +\frac{\nu}{2\tau}\E_k[\norm{\eta^k}^2]+2\ell_RC_0\Delta_k,
\end{aligned}
\end{equation*}
where $C_0$ is from Lemma~\ref{lem:B-bounds}.
\end{theorem}
\begin{proof}
Taking $v=\wh\eta^k$ in inequality~\eqref{eq:basic-Fk} and using the strong convexity of
$F_k$ on $T_{x^k}\M$, we obtain
\[ F_k(\eta^k)\leq  F_k(\wh\eta^k)+C_\Delta\Delta_k \leq F_k(0)-\frac{1}{2\alpha}\norm{\wh\eta^k}^2 +C_\Delta\Delta_k,\]
where $C_\Delta:=2\ell_RC_0$. Multiplying this inequality by $\gamma_k$ and
using the convexity of $R$, we have
\[R(x^k+\gamma_k\eta^k) \leq R(x^k)-\gamma_k\ip{m^k}{\eta^k} -\frac{\gamma_k}{2\alpha}\norm{\eta^k}^2 -\frac{\gamma_k}{2\alpha}\norm{\wh\eta^k}^2 +\gamma_k C_\Delta\Delta_k.\]
By Lemma~\ref{lem:step-size}, $\gamma_k\leq1$ and the choice of $\bar\alpha$,
Lemma~\ref{lem:retraction} yields
\[F(x^{k+1})\leq F(x^k)+\gamma_k\ip{\grad F(x^k)}{\eta^k}
+(L_F/2)\gamma_k^2\norm{\eta^k}^2\] and
\[R(x^{k+1})\leq R(x^k+\gamma_k\eta^k)
+L_R\gamma_k^2\norm{\eta^k}^2.\] Combining the preceding three inequalities
and using $(\gamma_k)^2\leq\gamma_k\leq1$, we get
\[\Phi(x^{k+1})\leq\Phi(x^k)-\sigma_0\gamma_k\norm{\eta^k}^2 -\frac{\gamma_k}{2\alpha}\norm{\wh\eta^k}^2 -\gamma_k\ip{e^k}{\eta^k}+C_\Delta\Delta_k.\]
By Young's inequality and $\gamma_k\leq1$, we have
\[-\gamma_k\ip{e^k}{\eta^k}\leq
(\tau/2)\gamma_k\norm{\eta^k}^2+(2\tau)^{-1}\norm{e^k}^2.\]
Taking conditional expectation with respect to $\calF_k$ and using
Assumption~\ref{ass:estimator}, we obtain the desired estimate.
\end{proof}
\section{Asymptotic convergence under the KL property}
In this section, we establish the full sequential convergence of iRPMVR under
the KL property. The analysis has two parts. We first prove an abstract KL
principle for descent relations with memory terms and summable tails, and then
apply it to a shifted objective on the tangent bundle. Throughout this section,
we work in the full-step regime $\gamma_k\equiv1$.

Some stochastic KL analyses use an expected-KL inequality to derive
finite-length estimates and whole-sequence convergence in expectation from
conditional expected descent. This inequality relates the expected objective
gap to the expected subdifferential distance through a single desingularizing
function. A key justification for this approach is Lemma~4.5 of
\cite{Driggs2021}, which claims that the ordinary KL property with a common
exponent implies such an expected-KL inequality with a desingularizing function
independent of the iteration. However, its proof is incorrect: the finite-sum
KL calculus used there yields, for each $k$, only an iteration-dependent
function $\varphi_k(s)=a_ks^{1-\theta}$, while the asserted uniform boundedness
of $\{a_k:k\in\N\}$ is not justified. The following example shows that this gap
is substantive: even a finite-length convergent stochastic process satisfying
the standard tail-free conditional expected descent and relative-error
conditions need not admit an iteration-independent power-type expected-KL
inequality for any exponent $\theta\in[0,1)$.

\begin{example}[Failure of uniform power-type expected-KL inequalities]
\label{ex:expected-KL}
Let $\Phi(x)=x^2$, $\Phi^*=0$, and $r\in(0,1)$. Let
$\{\xi_k:k\geq1\}$ be independent Bernoulli random variables with
$\mathbb P(\xi_k=1)=2^{-(2k-1)}$, and define $X^0=1$ and
$X^{k+1}=r\xi_{k+1}X^k$. Set
$\calF_k:=\sigma(\xi_1,\ldots,\xi_k)$ and
$D_k:=\abs{X^{k+1}-X^k}$. Writing
$p_{k+1}:=\mathbb P(\xi_{k+1}=1)$, direct calculation gives
\[\E_k[\Phi(X^{k+1})+D_k^2]
=\bigl(1-2p_{k+1}r(1-r)\bigr)\Phi(X^k)\leq\Phi(X^k) \quad \text{almost surely}.\]
Moreover, $D_k\geq(1-r)\abs{X^k}$, so
\[\dist(0,\partial\Phi(X^k))=2\abs{X^k}\leq2D_k/(1-r).\] Each sample path
either follows $r^k$ forever or eventually reaches zero; in either case,
$\sum_{k=0}^\infty D_k=1$ and $X^k\to0$. The function $\Phi$ is analytic and
semialgebraic with KL exponent $1/2$ at zero. Since
$\mathbb P(X^k=r^k)=\prod_{j=1}^k2^{-(2j-1)}=2^{-k^2}$, we have
\[\E[\Phi(X^k)-\Phi^*]=2^{-k^2}r^{2k}, \quad \E[\dist(0,\partial\Phi(X^k))]=2^{1-k^2}r^k.\]
Hence, for any $\theta\in[0,1)$, $a>0$, and
$\varphi(s)=as^{1-\theta}$, there holds 
\[\varphi'(\E[\Phi(X^k)-\Phi^*]) \E[\dist(0,\partial\Phi(X^k))] =2a(1-\theta)2^{-(1-\theta)k^2}r^{(1-2\theta)k}\to0.\]
Thus, no iteration-independent power-type desingularizing function, with any
exponent $\theta\in[0,1)$, satisfies the expected-KL inequality along this
sequence. In particular, the uniform same-exponent expected-KL inequality
constructed in the proof of Lemma~4.5 and subsequently used in the proofs of
Lemma~4.6 and Theorem~4.8 of \cite{Driggs2021} does not follow from the ordinary
KL property and the standard descent and relative-error conditions.
\end{example}

The example shows that the ordinary KL property, even together with standard
conditional expected descent and relative error, does not imply an
iteration-independent power-type expected-KL inequality for any exponent. We
therefore bypass expected-KL inequalities and instead apply the ordinary KL
inequality pathwise at the current random iterate before taking conditional
expectations, combining it with an augmented supermartingale argument. In its
tail-free specialization, the abstract principle below applies to
Example~\ref{ex:expected-KL} and yields
$\sum_{k=0}^\infty D_k<+\infty$ and $X^k\to0$ almost surely.

\subsection{A unified KL theorem with memory terms and summable tail perturbations}
The following theorem provides a KL principle for conditional expected descent
relations with multiple memory terms and summable tail perturbations. Part~(I) does not invoke the KL property: under conditional expected descent,
relative error, and continuity, it yields almost-sure convergence of the
objective values and stationarity of all cluster points. Part~(II) uses the ordinary pointwise KL property, together
with the deterministic limiting value and compactness assumptions, to upgrade
these conclusions to finite length and whole-sequence convergence.

\begin{theorem}[Unified stochastic KL principle]
	\label{thm:abstract-KL}
	Let $(\Omega,\mathcal F,\mathbb P)$ be a probability space with a filtration
	$\{\mathcal F_k:k\in\mathbb N\}$, and write
	$\mathbb E_k[\cdot]:=\mathbb E[\cdot\mid\mathcal F_k]$. Let \(\Psi:\mathbb R^d\to(-\infty,+\infty]\) be proper and
	lower semicontinuous. Assume that there exists a deterministic
	constant \(\underline\Psi\in\mathbb R\) such that
	\(\Psi(X^k)\ge\underline\Psi\) almost surely for every
	\(k\in\mathbb N\). Let $\{X^k:k\in\mathbb N\}\subset\dom\Psi$ be adapted, and
	let $\{\lambda_{\ell,k}:k\in\mathbb N\}$, $\ell=1,\ldots,m$, be nonnegative
	adapted sequences. Let $\{\sigma_k:k\in\mathbb N\}$ be nonnegative and
	deterministic, set
	\[
	S_k:=\sum_{j=k+1}^{\infty}\sigma_j, \quad 
	D_k:=\sum_{\ell=1}^m(\lambda_{\ell,k+1}+\lambda_{\ell,k}+\lambda_{\ell,k-1}),
	\]
	and assume that $\sum_{k=0}^{\infty}S_k^q<+\infty$ for some $q\in(0,1)$.
	Suppose that the quantities involved in the conditional expectations below are
	integrable and that, for all $k\ge k_0$, where $k_0$ is deterministic, the
	following conditions hold.
	
	{\rm (i)} {\rm(conditional expected descent)} There exist constants
	$a_{1,\ell}>0$ and $a_{2,\ell},a_{3,\ell}\in\mathbb R$, $\ell=1,\ldots,m$,
	such that $a_{2,\ell}+\max\{a_{3,\ell},0\}<a_{1,\ell}$ and
	\[
	\mathbb E_k[\Psi(X^{k+1})]
	\le
	\Psi(X^k)
	-\sum_{\ell=1}^m a_{1,\ell}\mathbb E_k[\lambda_{\ell,k+1}^2]
	+\sum_{\ell=1}^m a_{2,\ell}\lambda_{\ell,k}^2
	+\sum_{\ell=1}^m a_{3,\ell}\lambda_{\ell,k-1}^2
	+\sigma_{k+1}.
	\]
	
	{\rm (ii)} {\rm(Relative error)} There exists $C_{\rm re}>0$ such that
	\[
	\dist(0,\partial\Psi(X^k))\le C_{\rm re}(D_k+S_k^q)
	\]
	for all sufficiently large $k$.
	
	{\rm (iii)} {\rm(Continuity)}  \(\Psi\) is continuous relative to \(\dom\Psi\).
	
	Then the following statements hold:
	
	{\rm (I)} The sequence $\{\Psi(X^k):k\in\mathbb N\}$ converges almost surely to a
	finite random variable $\Psi_\infty$. Moreover, $S_k\to0$ and  \(\lambda_{\ell,k}\to0\) almost
	surely, \(\ell=1,\ldots,m.\) In addition, every cluster
	point of $\{X^k:k\in\mathbb N\}$ is almost surely stationary, that is, it
	belongs to \(\crit\Psi:=\{x\in\mathbb{R}^d: 0\in\partial \Psi (x)\}\).
	
	{\rm (II)} Assume further that the following conditions hold.
	\begin{itemize}
		\item[{\rm (iv)}] {\rm(Deterministic limiting value)} There exists a
		deterministic constant $\Psi_*$ such that $\Psi_\infty=\Psi_*$ almost surely.
		
		\item[{\rm (v)}] {\rm(Compactness)} There exists a deterministic compact set
		$S_\Psi\subset\mathbb R^d$ such that
		$\{X^k:k\in\mathbb N\}\subseteq S_\Psi$ almost surely.
		
		\item[{\rm (vi)}] {\rm(KL property)} Let
		\(
		\Omega_*:=\{x\in S_\Psi:\Psi(x)=\Psi_*,\ 0\in\partial\Psi(x)\}.
		\)
		Assume that $\Omega_*$ is nonempty and compact and that $\Psi$
		satisfies the KL property at every point of $\Omega_*$.
	\end{itemize}
	Then,
	\(\sum_{k=0}^{\infty}D_k<+\infty\) 
	almost surely. If, in addition, there exists $C_S>0$ such that
	\[\text{(Successive error)} \quad \|X^{k+1}-X^k\|\le C_S(D_{k+1}+S_k^q)
	\]
	holds for all
	sufficiently large $k$,  then
	\(
	\sum_{k=0}^{\infty}\|X^{k+1}-X^k\|<+\infty\) 
	almost surely. Consequently, $\{X^k:k\in\mathbb N\}$ converges almost surely
	to some $X^*\in\crit\Psi$.
\end{theorem}

\begin{proof}
	Set \(\lambda_{\ell,k}=0\) for \(k<0\). Choose
	\(\theta_{2,\ell}>\max\{a_{3,\ell},0\}\) such that
	\(a_{2,\ell}+\theta_{2,\ell}<a_{1,\ell}\), and then choose
	\[0<\theta_{1,\ell}\in(a_{2,\ell}+\theta_{2,\ell},a_{1,\ell}), \quad
	\ell=1,\ldots,m.\] Let \(s:=1/(1-q)>1\), so that
	\((s-1)/s=q\), and set \(\rho_k:=S_k^{1/s}\). Define
	\[
	\widehat\Psi(x,u,v,\rho)
	:=
	\Psi(x)+\sum_{\ell=1}^m\theta_{1,\ell}u_\ell^2
	+\sum_{\ell=1}^m\theta_{2,\ell}v_\ell^2+|\rho|^s,\]
	\[
	Y^k:=
	(X^k,\lambda_{1,k},\ldots,\lambda_{m,k},
	\lambda_{1,k-1},\ldots,\lambda_{m,k-1},\rho_k), \quad 
	\mathcal L_k:=\widehat\Psi(Y^k).
	\]
	Since \(S_k=\sigma_{k+1}+S_{k+1}\), it follows from \textup{(i)} that,
	for all \(k\ge k_0\), there holds
	\[
	\begin{aligned}
	\mathbb E_k[\mathcal L_{k+1}]
	\le\;&
	\mathcal L_k
	-\sum_{\ell=1}^m(a_{1,\ell}-\theta_{1,\ell})
	\mathbb E_k[\lambda_{\ell,k+1}^2]
	-\sum_{\ell=1}^m(\theta_{1,\ell}-a_{2,\ell}-\theta_{2,\ell})
	\lambda_{\ell,k}^2
	\\
	&\quad
	-\sum_{\ell=1}^m(\theta_{2,\ell}-a_{3,\ell})
	\lambda_{\ell,k-1}^2.
	\end{aligned}
	\]
	All coefficients on the right-hand side are positive. Since \(D_k^2\) is
	bounded by a constant multiple of
	\(\sum_{\ell=1}^m(\lambda_{\ell,k+1}^2+\lambda_{\ell,k}^2+
	\lambda_{\ell,k-1}^2)\), there exists \(c_0>0\) such that
	\begin{equation}\label{eq:L-descent}
		\mathbb E_k[\mathcal L_{k+1}]
		\le
		\mathcal L_k-c_0\mathbb E_k[D_k^2],
		\qquad k\ge k_0 .
	\end{equation}
	Applying the Robbins--Siegmund almost-supermartingale convergence theorem
	\cite[Theorem~1]{RobbinsSiegmund1971} to
	\(\mathcal L_k-\underline\Psi\), we obtain that
	\(\{\mathcal L_k:k\in\mathbb N\}\) converges almost surely and that
	\(\sum_k\lambda_{\ell,k}^2<+\infty\) almost surely for every
	\(\ell=1,\ldots,m\). Thus \(\lambda_{\ell,k}\to0\) almost surely. Since
	\(\sum_kS_k^q<+\infty\), we also have \(S_k\to0\). Therefore
	\[
	\Psi(X^k)
	=
	\mathcal L_k-\sum_{\ell=1}^m\theta_{1,\ell}\lambda_{\ell,k}^2
	-\sum_{\ell=1}^m\theta_{2,\ell}\lambda_{\ell,k-1}^2-S_k
	\]
	converges almost surely to a finite random variable, denoted by
	\(\Psi_\infty\). Moreover, \(D_k\to0\) almost surely, and hence \textup{(ii)}
	implies
	\(
	\dist(0,\partial\Psi(X^k))\to0
	\quad\hbox{almost surely}.
	\)
	Let \(\bar X\) be a cluster point of \(\{X^k:k\in\mathbb N\}\), say
	\(X^{k_j}\to\bar X\). By \textup{(iii)},
	\(\Psi(X^{k_j})\to\Psi(\bar X)\). Then, the closedness of the
	limiting-subdifferential graph
	\cite[Theorem~8.6]{RockafellarWets1998} yields
	\(0\in\partial\Psi(\bar X)\). This proves \textup{(I)}.
	
	We next prove \textup{(II)}. By \textup{(I)} and \textup{(iv)},
	\(\mathcal L_k\to\Psi_*\) almost surely. Moreover, by conditional Fatou's
	lemma and the supermartingale property \cite[Section~9.7(f),(i)]{Williams1991}, we obtain
	\[\mathcal L_k\ge\mathbb E_k[\liminf_j\mathcal L_j]=\Psi_* \quad \text{almost surely}.\]
	Set \(\Gamma_k:=\mathcal L_k-\Psi_*\ge0\) almost surely and
	\(
	\widehat\Omega_*:=\Omega_*\times\{0\}^{2m}\times\{0\}.
	\)
	By \textup{(v)}, \(\lambda_{\ell,k}\to0\) and \(\rho_k\to0\), we derive that
	\(\{Y^k:k\in\mathbb N\}\) is  bounded almost surely. Moreover, Item~\textup{(I)},
	conditions \textup{(iii)} and \textup{(v)} imply that any accumulation point
	of \(\{Y^k:k\in\mathbb N\}\) belongs to \(\widehat\Omega_*\) almost surely. Consequently, \( \dist(Y^k,\widehat\Omega_*)\to0\) almost surely.
	
	We next apply the KL property to \(\widehat\Psi\). Since \(\Psi\) satisfies
	the KL property on \(\Omega_*\), and the added functions
	\(r\mapsto\theta r^2\) and \(r\mapsto |r|^s\) satisfy the KL property at
	\(0\), the separable-sum rule for the generalized concave KL property
	\cite[Theorem~3.8]{WangWang2023} implies that \(\widehat\Psi\) satisfies the
	KL property at every point of \(\widehat\Omega_*\). Indeed, each desingularizing function
	\(\varphi_i\) involved in this rule may be replaced by
	\(\varphi_i(t)+\sqrt{t}\), which remains a desingularizing function and
	is strictly concave. By the uniformized KL property
	Lemma~\ref{lem:uniform-KL}, there exist \(\epsilon>0\), \(\delta>0\),
	and a concave desingularizing function \(\widehat\varphi\) such that
	\(
	\widehat\varphi'(\widehat\Psi(Y)-\Psi_*)
	\dist(0,\partial\widehat\Psi(Y))\ge1
	\)
	whenever
	\[
	Y\in\mathcal U:=
	\{Y:\dist(Y,\widehat\Omega_*)<\delta,\
	0\le\widehat\Psi(Y)-\Psi_*<\epsilon\}.
	\]
	We extend \(\widehat\varphi\) constantly to \([\epsilon,+\infty)\); this is
	only a notational convention.
	
	For all sufficiently large \(k\), \textup{(ii)} and the definition of
	\(\widehat\Psi\) imply that there exists \(\hat C>0\) such that
	\begin{equation}\label{eq:lifted-relative}
		\dist(0,\partial\widehat\Psi(Y^k))
		\le \hat C(D_k+S_k^q).
	\end{equation}
	On the event \({\Gamma_k=0}\), inequality~\eqref{eq:L-descent}, together with \(\Gamma_{k+1}\geq0\), implies
	\(
	\mathbb E_k[D_k^2]=0.
	\)
	We therefore consider the event \({\Gamma_k>0}\).	
	By conditional Jensen's inequality, the concavity of \(\widehat\varphi\), and
	\eqref{eq:L-descent}, we obtain
	\[
	\widehat\varphi(\Gamma_k)
	-\mathbb E_k[\widehat\varphi(\Gamma_{k+1})]
	\ge
	\widehat\varphi(\Gamma_k)
	-\widehat\varphi(\mathbb E_k[\Gamma_{k+1}])
	\ge
	c_0\widehat\varphi'(\Gamma_k)\mathbb E_k[D_k^2].
	\]
	On the event \({Y^k\in \mathcal U}\cap\{\Gamma_k>0\}\), the KL inequality and \eqref{eq:lifted-relative} yield
	\(1\le \hat C\widehat\varphi'(\Gamma_k)(D_k+S_k^q)\). Hence
	\(
	D_k\le 2\hat C\widehat\varphi'(\Gamma_k)D_k^2+S_k^q .
	\)
	Indeed, if \(D_k\ge S_k^q\), then \(D_k+S_k^q\le2D_k\), and hence
	\(D_k\le2\hat C\widehat\varphi'(\Gamma_k)D_k^2\); while if
	\(D_k<S_k^q\), the estimate is immediate. Taking conditional expectations and using the preceding
	estimate, we obtain, whenever \(Y^k\in\mathcal U\),
	\begin{equation}\label{eq:finite-length-step}
		\mathbb E_k[D_k]
		\le
		\frac{2\hat C}{c_0}
		\bigl(\widehat\varphi(\Gamma_k)
		-\mathbb E_k[\widehat\varphi(\Gamma_{k+1})]\bigr)
		+S_k^q .
	\end{equation}
	
	It remains to discuss when \(Y^k\in\mathcal U\). For each deterministic
	\(K\in\mathbb N\), define
	\(
	\tau_K:=\inf\{k\ge K:Y^k\notin\mathcal U\}\)  and \(\inf\emptyset:=+\infty .
	\)
	Since \(\dist(Y^k,\widehat\Omega_*)\to0\) and \(\Gamma_k\to0\) almost surely,
	one has
	\[
	\mathbb P\Bigl(\bigcup_{K=1}^{\infty}\{\tau_K=+\infty\}\Bigr)=1.
	\]
	Fix \(K\), and set\[
	I_k^K:=
	\begin{cases}
	1, & \text{if } k<\tau_K,\\
	0, & \text{otherwise}.
	\end{cases}
	\]
	Multiplying \eqref{eq:finite-length-step} by \(I_k^K\), taking expectations, and summing
	from \(k=K\) to \(K'\), we get, for some \(c>0\),
	\[
	\sum_{k=K}^{K'}\mathbb E[I_k^KD_k]
	\le
	c\sum_{k=K}^{K'}S_k^q
	+c\sum_{k=K}^{K'}\mathbb E[I_k^KA_k],
	\]
	where
	\(
	A_k:=
	\widehat\varphi(\Gamma_k)
	-\mathbb E_k[\widehat\varphi(\Gamma_{k+1})]\ge0.
	\)
	Since \(I_k^K\le I_{k-1}^K\) for \(k>K\), we have
	\[
	\sum_{k=K}^{K'}\mathbb E[I_k^KA_k]
	\le \mathbb E[\widehat\varphi(\Gamma_K)]<+\infty .
	\]
	Moreover, \(\sum_kS_k^q<+\infty\). Letting \(K'\to\infty\), we obtain
	\(
	\sum_{k=K}^{\infty}I_k^KD_k<+\infty
	\quad\hbox{almost surely}.
	\)
	On the event \(\{\tau_K=+\infty\}\), one has \(I_k^K=1\) for all \(k\ge K\),
	and hence
	\(
	\sum_{k=K}^{\infty}D_k<+\infty .
	\)
	Since
	\(
	\mathbb P\Bigl(\bigcup_{K=1}^{\infty}\{\tau_K=+\infty\}\Bigr)=1,
	\)
	we obtain \(\sum_{k=0}^{\infty}D_k<+\infty\) almost surely.
	
	Finally, assume the successive-error bound. Since
	\(\sum_kS_k^q<+\infty\) and \(\sum_kD_k<+\infty\) almost surely, we have
	\(
	\sum_{k=0}^{\infty}\|X^{k+1}-X^k\|<+\infty\) almost surely.
	Thus \(\{X^k: k\in\mathbb{N}\}\) is almost surely Cauchy and converges to some
	\(X^*\in S_\Psi\). We then get the desired \textup{(II)} by applying
	\textup{(I)}. This completes the proof.
\end{proof}

\begin{remark}\label{rem:det-limit}
The deterministic-limit condition in Item~(II)(iv) is automatic in the
deterministic setting. In the stochastic setting, it holds if there exist a
deterministic compact set $K\subseteq S_\Psi$ and a constant $\Psi_*\in\R$
such that $X^k\in K$ for all sufficiently large $k$ almost surely,
$\Psi(x)=\Psi_*$ for all $x\in\crit\Psi\cap K$. Indeed, Item~(I) ensures that
$\Psi(X^k)$ converges almost surely and that every cluster point is stationary.
The compactness of $K$ and condition~(iii) therefore imply that the limit equals
$\Psi_*$ almost surely. The common-critical-value condition holds, for example,
if $\Psi$ is convex, if $\Psi$ has a unique stationary point in $K$, or if, for
some $C,p>0$,
\[\Psi(x)-\inf_{y\in K}\Psi(y)\leq C\dist^p(0,\partial\Psi(x)), \quad x\in K.\]
Global gradient-dominance conditions of this type have been used in the
analysis of smooth stochastic nonconvex optimization; see, e.g.,
\cite{Fatkhullin2022,Karandikar2024}. The condition above is its
limiting-subdifferential counterpart for the present nonsmooth setting.
\end{remark}

The following two elementary lemmas provide the tail and discrete comparison
estimates used in the subsequent KL analysis. For Lemma~\ref{lem:tail}, the integral
comparison gives
\(
\sum_{j=k+1}^{\infty}\sigma_j=O(k^{1-r}),
\)
so that its \(q\)-th power is summable whenever \(q(r-1)>1\).
For Lemma~\ref{lem:comparison}, a standard discrete comparison argument distinguishes whether
the perturbation term \(t^{-r}\) or the descent term \(a_t^{\vartheta}\)
dominates. These two cases yield, respectively,
\(a_t=O(t^{-r/\vartheta})\) and
\(a_t=O(t^{-1/(\vartheta-1)})\).
We therefore omit the routine proofs.

\begin{lemma}\label{lem:tail}
Suppose that there exist constants $C_\sigma>0$, $r>2$, and an integer
$t_0\geq1$ such that $0\leq\sigma_t\leq C_\sigma t^{-r}$ for $t\geq t_0$.
Then
$\sum_{k=0}^\infty\left(\sum_{j=k+1}^\infty\sigma_j\right)^q<+\infty \ \text{for any }q\in\left(\frac1{r-1},1\right).$
\end{lemma}

\begin{lemma}[A perturbed discrete comparison estimate]\label{lem:comparison}
Let $r>2$, $\beta_1>0$, $\beta_2\geq0$, and let
$\{a_t:t\in\N\}$ be a nonnegative sequence satisfying, for all sufficiently
large $t$,
\[\beta_1a_t^\vartheta\leq a_{t-1}-a_t+\beta_2t^{-r}.\]
If $\vartheta\in(1,2)$, then there exists a constant $C>0$ such that, for all sufficiently large $t$, 
\[a_t\leq Ct^{-1/(\vartheta-1)}+Ct^{-r/\vartheta}.\]
\end{lemma}

Pointwise KL rates require a pathwise descent relation, and hence we only
record the deterministic consequence below.

\begin{theorem}[KL rates in the deterministic case]\label{thm:abstract-rates}
Assume that the deterministic counterparts of conditions~(i)--(vi) in
Theorem~\ref{thm:abstract-KL} hold. Assume further that there exists $C_S>0$
such that, for all sufficiently large $k$,
$\norm{x^{k+1}-x^k}\leq C_S(D_{k+1}+S_k^q)$, and that the descent relation in
Theorem~\ref{thm:abstract-KL}(i) holds pathwise, namely,
\begin{equation}\label{eq:pathwise-descent}
 \Psi(x^{k+1})\leq\Psi(x^k)
 -\sum_{\ell=1}^m a_{1,\ell}\lambda_{\ell,k+1}^2
 +\sum_{\ell=1}^m a_{2,\ell}\lambda_{\ell,k}^2
 +\sum_{\ell=1}^m a_{3,\ell}\lambda_{\ell,k-1}^2+\sigma_{k+1}.
\end{equation}
Let $x^*$ be the limit point obtained in Theorem~\ref{thm:abstract-KL}(II), and
assume that $\Psi$ has KL exponent $\theta\in[0,1)$ at $x^*$. Then the following
assertions hold.
\begin{enumerate}
\item[(i)] Suppose that $\theta\in[0,1/2]$ and
$0\leq\sigma_k\leq C_\sigma\varrho^k$ for all sufficiently large $k$, where
$C_\sigma\geq0$ and $\varrho\in(0,1)$. Then there exist constants $C>0$ and
$\rho\in(0,1)$ such that, for all sufficiently large $k$,
\[\abs{\Psi(x^k)-\Psi(x^*)}\leq C\rho^k, \quad D_k\leq C\rho^k, \quad \norm{x^k-x^*}\leq C\rho^k.\]
\item[(ii)] Suppose that $0\leq\sigma_k\leq C_\sigma k^{-r}$ for all sufficiently
large $k$, where $C_\sigma\geq0$ and $r>2$. 
Define
\[
\bar\theta:=
\begin{cases}
\max\left\{\theta,\dfrac{r}{2(r-1)}\right\}, & \text{if } C_\sigma>0,\\[1mm]
\theta, & \text{if } C_\sigma=0.
\end{cases}
\] If $\bar\theta\in(1/2,1)$, then there
exists a constant $C>0$ such that, for all sufficiently large $k$,
\[\abs{\Psi(x^k)-\Psi(x^*)}\leq Ck^{-1/(2\bar\theta-1)}+Ck^{-r/(2\bar\theta)}, \quad D_k\leq Ck^{-1/(4\bar\theta-2)}+Ck^{-r/(4\bar\theta)}.\]
Moreover, for some $C_*>0$, $\norm{x^k-x^*}\leq C_*k^{-\kappa}$, where
\[
\kappa:=
\begin{cases}
\dfrac{1-\theta}{2\theta-1},
& \text{if } C_\sigma=0,\\[2mm]
\displaystyle
\min\left\{
\dfrac{1-\bar\theta}{2\bar\theta-1},
\dfrac{r(1-\bar\theta)}{2\bar\theta},
\dfrac{r-2}{2}
\right\},
& \text{if } C_\sigma>0.
\end{cases}
\]
\end{enumerate}
\end{theorem}

\begin{proof}
Set $\Psi_*:=\Psi(x^*)$. As in the proof of
Theorem~\ref{thm:abstract-KL}, choose $\theta_{1,\ell}$ and
$\theta_{2,\ell}$. For a chosen $q\in(0,1)$, define
$S_k,s,y^k,\widehat\Psi$ as in the proof of Theorem~\ref{thm:abstract-KL}. Let
$\Gamma_k:=\widehat\Psi(y^k)-\Psi_*$. Since $\{\Gamma_k:k\in\N\}$ is
nonincreasing for all sufficiently large $k$ and converges to zero, the definition of $\widehat\Psi$ implies
\[\abs{\Psi(x^k)-\Psi_*}\leq\Gamma_k+\widehat\Psi(y^k)-\Psi(x^k)
\leq\Gamma_k+C_\Psi D_k^2+S_k\] for some $C_\Psi>0$. The pathwise descent
relation~\eqref{eq:pathwise-descent} and the construction in the proof of
Theorem~\ref{thm:abstract-KL} yield $\Gamma_k-\Gamma_{k+1}\geq c_0D_k^2$ for
all sufficiently large $k$. Moreover, the relative-error condition leads to
\[\dist(0,\partial\widehat\Psi(y^k))\leq C_{\rm re}(D_k+S_k^q).\]
We prove (i) first. The geometric bound on $\sigma_k$ implies
$\sum_{k=0}^\infty S_k^q<+\infty$ for every $q\in(0,1)$, so we may take
$q=1/2$. Since $\theta\leq1/2$, the lifted function $\widehat\Psi$ has KL
exponent $1/2$ at $(x^*,0,\ldots,0)$
\cite[Theorem~3.3]{LiPong2018}. Hence, for all sufficiently large $k$,
$\Gamma_k^{1/2}\leq C_{\rm KL,1}(D_k+S_k^{1/2})$ for some
$C_{\rm KL,1}>0$. Together with the descent estimate, this yields
\[\Gamma_k\leq C_{\Gamma,1}(\Gamma_k-\Gamma_{k+1})+C_{\Gamma,2}S_k\]
for some $C_{\Gamma,1},C_{\Gamma,2}>0$. Since $S_k=0$ eventually if
$C_\sigma=0$, and $S_k\leq C_{S,\varrho}\varrho^k$ otherwise, the preceding
inequality implies \[\Gamma_{k+1}\leq a_\Gamma\Gamma_k+C_{\Gamma,\varrho}\varrho^k\]
for some $a_\Gamma\in(0,1)$ and $C_{\Gamma,\varrho}>0$. Hence
$\Gamma_k\leq C_\Gamma\rho^k$ for some $C_\Gamma>0$ and
$\rho\in(0,1)$. The estimates for $\Gamma_k$ and $D_k$, together with
$S_k=O(\varrho^k)$, yield the claimed estimate for
$\abs{\Psi(x^k)-\Psi_*}$. The estimate for $D_k$ follows from
$D_k^2\leq c_0^{-1}(\Gamma_k-\Gamma_{k+1})$, after enlarging $\rho$ if
necessary. Finally, the successive-error bound leads to
\[\norm{x^k-x^*}\leq\sum_{j=k}^\infty\norm{x^{j+1}-x^j}
\leq C_S\sum_{j=k}^\infty(D_{j+1}+S_j^{1/2})\leq C_x\rho^k\]
for some $C_x>0$.

We next prove (ii). When $C_\sigma>0$, choose
$q=r/[2(r-1)]$. By Lemma~\ref{lem:tail},
$S_k^q\leq C_{S,q}C_\sigma^qk^{-r/2}$ for some $C_{S,q}>0$, and
$\widehat\Psi$ has KL exponent $\bar\theta=\max\{\theta,r/[2(r-1)]\}$
\cite[Theorem~3.3]{LiPong2018}. When
$C_\sigma=0$, take $q=1/2$; then $S_k=0$ for all sufficiently large $k$, and
the KL exponent of $\widehat\Psi$ is $\bar\theta=\max\{\theta,1/2\}$. Thus, in both
cases, $\Gamma_k^{\bar\theta}\leq C_{\rm KL,2}(D_k+S_k^q)$ for some
$C_{\rm KL,2}>0$. Combining this with the descent estimate, we obtain
\[\Gamma_k-\Gamma_{k+1}
\geq c_\Gamma\Gamma_k^{2\bar\theta}
-C_{\rm pert}C_\sigma^{r/(r-1)}k^{-r}\]
for some $c_\Gamma,C_{\rm pert}>0$. Applying Lemma~\ref{lem:comparison} yields the
claimed rate for $\Gamma_k$. Moreover, we have 
\[\widehat\Psi(y^k)-\Psi(x^k)\leq C_\Psi D_k^2+S_k, \quad S_k=O(k^{1-r})=O(k^{-r/(2\bar\theta)}).\] Hence the same rate holds for
$\abs{\Psi(x^k)-\Psi(x^*)}$. The estimate for $D_k$ follows from
\[D_k^2\leq c_0^{-1}(\Gamma_k-\Gamma_{k+1})\leq c_0^{-1}\Gamma_k.\]
Finally, the KL finite-length estimate
$\sum_{j=k}^\infty D_{j+1}\leq C\Gamma_k^{1-\bar\theta}
+C\sum_{j=k}^\infty S_j^q$, together with the successive-error bound, yields
the claimed estimate for $\norm{x^k-x^*}$. This completes the proof.
\end{proof}
\subsection{Shifted objective on the tangent bundle}
Since the proximal step is defined in the tangent space, we use the shifted
objective
\begin{equation}\label{eq:shifted-objective}
 \Phi_0(x,\eta):=\Phi(x+\eta)+\ind_{\TM}(x,\eta),
 \qquad (x,\eta)\in\R^d\times\R^d.
\end{equation}
This choice directly incorporates the tangent-space proximal step into the
objective, while \((x,0)\) being critical for \(\Phi_0\) implies the
first-order criticality of \(x\) for the original problem.
The next estimate gives a conditional expected descent relation for $\Phi_0$.
\begin{proposition}[Conditional expected shifted objective descent]
\label{prop:shifted-descent}
Suppose Assumption~\ref{ass:geometry} holds. Let $\{x^k:k\in\N\}$ be generated
by Algorithm~\ref{alg:iRPMVR} in the full-step regime $\gamma_k\equiv1$, and
let $\wh\eta^k$ be defined by \eqref{eq:exact-step}. Let $\rho_1,\rho_2>0$ and
assume that $0<\bar\alpha<\min\{C_\alpha,1/\ell_f\}$, where $C_\alpha$ is from
Lemma~\ref{lem:step-size}. Set
\[a:=\frac{1}{2\alpha}-\frac{\ell_f}{2}-\frac{\rho_2}{2}, \quad c_\eta:=\frac{1}{4\alpha}+\ell_f^2\rho_1+2L_\Phi, \quad C_{\rm sh}:=\frac{C_1^2}{\alpha}+2L_\Phi C_1^2+\frac{C_1^2}{2\rho_1}+2C_0\ell_R,\]
where $C_0$ is from Lemma~\ref{lem:B-bounds} and $C_1$ is from
Lemma~\ref{lem:primal-error}. Then the following statements hold.
\begin{enumerate}
\item[(i)] For every $k\geq0$,
\[
\begin{aligned}
\Phi_0(x^{k+1},\wh\eta^{k+1})\leq{}&\Phi_0(x^k,\wh\eta^k)
+c_\eta\norm{\wh\eta^k}^2-a\norm{\wh\eta^{k+1}}^2 \\
&+\rho_1\norm{e^k}^2+\frac1{2\rho_2}\norm{e^{k+1}}^2
+C_{\rm sh}\Delta_k.
\end{aligned}
\]
\item[(ii)] If Assumption~\ref{ass:estimator} also holds and
$0<\bar\alpha<\min\{C_\alpha,1/(\ell_f+\rho_2)\}$, then, for any
$\omega\in(0,a)$ and every $k\geq0$ we have
\[
\begin{aligned}
&\E_{k+1}[\Phi_0(x^{k+1},\wh\eta^{k+1})] \\
&\quad\leq \Phi_0(x^k,\wh\eta^k)
-\omega\E_{k+1}[\norm{\wh\eta^{k+1}}^2]
+c_\eta\norm{\wh\eta^k}^2 \\
&\qquad-\left(\frac{a-\omega}{2}-\frac{\nu}{2\rho_2}\right)
\E_{k+1}[\norm{\eta^{k+1}}^2]
+\rho_1\norm{e^k}^2 \\
&\qquad+\frac{c_e-\mu}{2\rho_2}\mathcal E_{k+1}
-\frac{c_e}{2\rho_2}\E_{k+1}[\mathcal E_{k+2}] \\
&\qquad+C_{\rm sh}\Delta_k
+(a-\omega)C_1^2\Delta_{k+1}.
\end{aligned}
\]
\end{enumerate}
\end{proposition}
\begin{proof}
Fix $k\geq0$ and write \[x:=x^k, \quad  \eta:=\eta^k, \quad \wh\eta:=\wh\eta^k, \quad m:=m^k, \quad e:=e^k.\]
By
\eqref{eq:basic-Fk} with $v=\wh\eta$, the Lipschitz continuity of $\nabla F$,
and $\alpha<1/\ell_f$, we have
\begin{equation*}
\Phi(x+\eta)\leq\Phi(x+\wh\eta)
+\ip{\nabla F(x+\wh\eta)-m}{\eta-\wh\eta}
+\frac1{2\alpha}\bigl(\norm{\wh\eta}^2-\norm\eta^2\bigr)
+2C_0\ell_R\Delta_k.
\end{equation*}
Since $\eta-\wh\eta\in T_x\M$ and $e=m-\grad F(x)$, Young's inequality and
Lemma~\ref{lem:primal-error} yield
\[\ip{\nabla F(x+\wh\eta)-m}{\eta-\wh\eta}
\leq\ell_f^2\rho_1\norm{\wh\eta}^2+\rho_1\norm e^2
+\frac{C_1^2}{2\rho_1}\Delta_k,\] and \( \norm{\wh\eta}^2-\norm\eta^2\leq\tfrac12\norm{\wh\eta}^2+2C_1^2\Delta_k.\)
Using Lemma~\ref{lem:retraction}, we obtain
\(\Phi(x^{k+1})\leq\Phi(x^k+\eta^k)+L_\Phi\norm{\eta^k}^2.\) On the other
hand, the optimality of $\wh\eta^{k+1}$ in \eqref{eq:exact-step}, the Lipschitz
continuity of $\nabla F$, and Young's inequality yield
\[\Phi(x^{k+1}+\wh\eta^{k+1}) \leq\Phi(x^{k+1})-a\norm{\wh\eta^{k+1}}^2 +\frac1{2\rho_2}\norm{e^{k+1}}^2.\]
Combining the above inequalities leads to Item~(i).

We next prove Item~(ii). Under the stepsize bound in Item~(ii), we have $a>0$.
Fix $\omega\in(0,a)$. By Lemma~\ref{lem:primal-error},
$\norm{\wh\eta^{k+1}}^2\geq\tfrac12\norm{\eta^{k+1}}^2-C_1^2\Delta_{k+1}$.
Hence
\[-a\norm{\wh\eta^{k+1}}^2 \leq-\omega\norm{\wh\eta^{k+1}}^2 -\frac{a-\omega}{2}\norm{\eta^{k+1}}^2 +(a-\omega)C_1^2\Delta_{k+1}.\]
Taking conditional expectation $\E_{k+1}[\cdot]$ in Item~(i), using the above
estimate and the $\mathcal F_{k+1}$-measurability of $\Delta_{k+1}$, and applying
Assumption~\ref{ass:estimator} at the index $k+1$, we obtain Item~(ii).
\end{proof}

We next establish the relative-error estimate for the shifted objective
$\Phi_0$ required by the abstract KL principle.

\begin{proposition}[Relative error for the shifted objective]
\label{prop:shifted-relative}
Suppose Assumption~\ref{ass:geometry} holds. Let $\Phi_0$ be defined by
\eqref{eq:shifted-objective}, and let $\{x^k:k\in\N\}$ be generated by Algorithm~\ref{alg:iRPMVR} in the full-step regime $\gamma_k\equiv1$, with
$\wh\eta^k$ defined by \eqref{eq:exact-step}. Then there exists a constant
$c_\Phi>0$, independent of $k$ and $n$, such that
\[\dist(0,\partial\Phi_0(x^k,\wh\eta^k)) \leq
c_\Phi(\norm{\wh\eta^k}+\norm{e^k}), \quad k\in\N.\]
\end{proposition}

\begin{proof}
Fix $k\in\N$, and write $x:=x^k$, $\eta:=\wh\eta^k$,
$m:=m^k$, and $e:=e^k=m^k-\grad F(x^k)$. Since $\eta$ is the exact minimizer
of \eqref{eq:exact-step}, there exist $\xi\in\partial R(x+\eta)$ and
$\zeta'\in N_x\M$ such that
$m+\alpha^{-1}\eta+\xi+\zeta'=0$. Since
$\grad F(x)-\nabla F(x)\in N_x\M$, there exists $\zeta\in N_x\M$ such that
\begin{equation}\label{eq:shifted-optimality}
 \nabla F(x)+e+\alpha^{-1}\eta+\xi+\zeta=0.
\end{equation}
By Lemma~\ref{lem:normal-lift}, there exists $u\in\R^d$ such that
$(u,\zeta)\in N_{\TM}(x,\eta)$ and
$\norm{u-\zeta}\leq C_{\rm tb}\norm\eta\norm\zeta$. From
\eqref{eq:shifted-optimality}, the boundedness of $\nabla F$ on the compact
manifold $\M$, and the Lipschitz continuity of $R$, we have
$\norm\zeta\leq C(1+\norm e+\norm\eta)$. Together with
Lemma~\ref{lem:step-size} and \eqref{eq:m-bounds}, this implies
\begin{equation}\label{eq:u-zeta}
 \norm{u-\zeta}\leq C(\norm\eta+\norm e),
\end{equation}
where $C>0$ is independent of $k$ and $n$.

Set $v_1:=\nabla F(x+\eta)+\xi+u$ and
$v_2:=\nabla F(x+\eta)+\xi+\zeta$. By the subdifferential sum rule applied to
\eqref{eq:shifted-objective}, we have $(v_1,v_2)\in\partial\Phi_0(x,\eta)$.
Moreover, by \eqref{eq:shifted-optimality}, we have
\[v_2=\nabla F(x+\eta)-\nabla F(x)-e-\alpha^{-1}\eta, \quad v_1=v_2+(u-\zeta).\] Using the Lipschitz continuity of $\nabla F$ and
\eqref{eq:u-zeta}, we obtain the desired result.
\end{proof}

\subsection{Asymptotic sequential convergence}
We now combine the shifted descent estimate and the relative error bound to
establish the sequential convergence of iRPMVR. We first introduce an implementable
rule for the inner tolerance $\Delta_k$, and then show that every cluster point
generated under this rule is critical for the original problem. This prepares
the application of the abstract KL theorem.

\begin{assumption}[Implementable inner tolerance for sequential convergence]
\label{ass:sequential-tolerance}
There exist constants $\kappa_1,\kappa_2\geq0$, $k_0\in\N$ and $r>2$ such
that
\[\Delta_k\leq\min\left\{\bar\Delta, \max\{\kappa_1\norm{\eta^{k-1}}^2,\kappa_2(k-k_0)^{-r}\}\right\}, \quad k\geq k_0+1.\]
\end{assumption}

Under Assumption~\ref{ass:sequential-tolerance}, we have 
\[\Delta_k\leq\kappa_1\norm{\eta^{k-1}}^2+\kappa_2(k-k_0)^{-r}, \quad k\geq k_0+1.\] In
particular, the deterministic tail $(k-k_0)^{-r}$ is summable, and since
$r>2$, it also satisfies the tail condition required in
Theorem~\ref{thm:abstract-KL}.

We first derive subsequential stationarity from the conditional expected
descent estimate.

\begin{lemma}[Subsequential stationarity]\label{lem:subsequence}
Suppose Assumptions~\ref{ass:geometry}, \ref{ass:estimator}, and
\ref{ass:sequential-tolerance} hold. Let $\Phi_0$ be defined by
\eqref{eq:shifted-objective}, and let $\{x^k:k\in\N\}$ be generated by
Algorithm~\ref{alg:iRPMVR} in the full-step regime $\gamma_k\equiv1$, with
$\wh\eta^k$ defined by
\eqref{eq:exact-step}. Let $\rho_1,\rho_2>0$, and let $a,c_\eta,C_{\rm sh}$ be
the constants in Proposition~\ref{prop:shifted-descent}. Choose
$\chi_e>\rho_1$. Assume that
\begin{equation*}
0<\bar\alpha<\min\{C_\alpha,1/(\ell_f+\rho_2),
1/(2\ell_f+2\rho_2+4\ell_f^2\rho_1+8L_\Phi+8\nu/\rho_2+16\chi_e\nu)\},
\end{equation*}
where $C_\alpha$ is from Lemma~\ref{lem:step-size}. Set $d_\alpha:=a-c_\eta$,
and choose $\kappa_1\geq0$ such that
\[0\leq\kappa_1<
\frac{d_\alpha/4-\nu/(2\rho_2)-\chi_e\nu}
{(d_\alpha C_1^2)/2+C_{\rm sh}},\] where $C_1$ is from
Lemma~\ref{lem:primal-error}. Then,
$\{\Phi_0(x^k,\wh\eta^k):k\in\N\}$ converges almost surely to a finite
random variable $\Phi_\infty$. Moreover, almost surely,
$\wh\eta^k\to0, \eta^k\to0, e^k\to0, \mathcal E_k\to0.$
Consequently, every accumulation point of $\{x^k:k\in\N\}$ is a critical point
of problem~\eqref{eq:model} almost surely.
\end{lemma}

\begin{proof}
Set \[z^k:=(x^k,\wh\eta^k),\quad \Psi_e(z,e):=\Phi_0(z)+\chi_e\norm e^2.\] For simplicity, we define
$\mathcal L_k:=\Psi_e(z^k,e^k)$. By the definitions of $a$ and $c_\eta$,
\[d_\alpha=1/(4\alpha)-\ell_f/2-\rho_2/2-\ell_f^2\rho_1-2L_\Phi.\]
The stepsize bound leads to $d_\alpha>2\nu/\rho_2+4\chi_e\nu$. Set
$\omega:=(a+c_\eta)/2$. Then $\omega>c_\eta$ and $a-\omega=d_\alpha/2$.
Applying Proposition~\ref{prop:shifted-descent}(ii) with this $\omega$, adding
$\chi_e\E_{k+1}[\norm{e^{k+1}}^2]$ to both sides, using
Assumption~\ref{ass:estimator} at the index $k+1$, we obtain by using
Assumption~\ref{ass:sequential-tolerance} and enlarging the deterministic tail
constant if necessary, for all sufficiently large $k$,
\[
\begin{aligned}
\E_{k+1}[\mathcal L_{k+1}]
\leq{}& \mathcal L_k
-\omega\E_{k+1}[\norm{\wh\eta^{k+1}}^2]
+c_\eta\norm{\wh\eta^k}^2 \\
 &-\left(d_\alpha/4-\nu/(2\rho_2)-\chi_e\nu\right)
 \E_{k+1}[\norm{\eta^{k+1}}^2] \\
 &+\frac{d_\alpha C_1^2\kappa_1}{2}\norm{\eta^k}^2
 +C_{\rm sh}\kappa_1\norm{\eta^{k-1}}^2 \\
&
-\left(1/(2\rho_2)+\chi_e\right)c_e
\E_{k+1}[\mathcal E_{k+2}]+\left(1/(2\rho_2)+\chi_e\right)(c_e-\mu)
\mathcal E_{k+1}\\
&-(\chi_e-\rho_1)\norm{e^k}^2
+\sigma_{k+1}, 
\end{aligned}
\] where
$\sigma_{k+1}:=C_\sigma\kappa_2(k+1-k_0)^{-r}$ for some $C_\sigma>0$.
Since $r>2$, Lemma~\ref{lem:tail} yields the required summable-tail condition.

Lemma~\ref{lem:step-size} and the continuity of $\Phi$ imply that
$\{\Phi_0(z^k):k\in\N\}$, and hence
$\{\Psi_e(z^k,e^k):k\in\N\}$, is bounded from below by a
deterministic constant. We apply Theorem~\ref{thm:abstract-KL}(I) with the
shifted filtration $\mathcal G_k:=\calF_{k+1}$, so that
$\E[\cdot\mid\mathcal G_k]=\E_{k+1}[\cdot]$, to
$X^k:=(z^k,e^k)$ and $\Psi:=\Psi_e$, with memory terms
\[\lambda_{1,k}:=\norm{\wh\eta^k}, \quad \lambda_{2,k}:=\norm{\eta^k}, \quad \lambda_{3,k}:=\sqrt{\mathcal E_{k+1}}, \quad 
\lambda_{4,k}:=\norm{e^{k-1}}.\] The bound of $\bar\alpha$
and the choice of $\kappa_1$ imply the condition that
$a_{2,\ell}+\max\{a_{3,\ell},0\}<a_{1,\ell}$ for $\ell=1,\ldots,4$ in
Theorem~\ref{thm:abstract-KL} holds. Proposition~\ref{prop:shifted-relative}
yields that
\[\dist(0,\partial\Psi_e(z^k,e^k))\leq C(\norm{\wh\eta^k}+\norm{e^k}).\]
Hence the relative-error condition (ii) in Theorem~\ref{thm:abstract-KL} holds.
Let $x^{k_j}\to x^*$ be an arbitrary convergent subsequence. Then,
Theorem~\ref{thm:abstract-KL}(I) yields
\[\wh\eta^k\to0, \quad \eta^k\to0, \quad \mathcal E_k\to0, \quad e^k\to 0 \quad  \text{almost surely}.\]
Moreover, $(x^*,0)$ is a critical point of $\Phi_0$. Hence, by
\eqref{eq:shifted-objective}, there exist $\xi^*\in\partial R(x^*)$ and
$(u^*,\zeta^*)\in N_{\TM}(x^*,0)$ such that
\[0=(\nabla F(x^*)+\xi^*+u^*,\nabla F(x^*)+\xi^*+\zeta^*).\]
Since $(u^*,\zeta^*)\in N_{\TM}(x^*,0)$, we have
$\zeta^*\in N_{x^*}\M$. Thus
$0\in\nabla F(x^*)+\partial R(x^*)+N_{x^*}\M$, which means that $x^*$ is a
critical point of problem~\eqref{eq:model}. This completes the proof.
\end{proof}

We are now ready to show the full sequence convergence of
Algorithm~\ref{alg:iRPMVR}. The KL assumption below is imposed only on the
shifted objective $\Phi_0$.

\begin{theorem}[Sequential convergence under the KL property]
\label{thm:sequential}
Suppose that the assumptions and parameter conditions in
Lemma~\ref{lem:subsequence} hold. Let $z^k:=(x^k,\wh\eta^k)$, and assume that
$\Phi_\infty$ in Lemma~\ref{lem:subsequence} equals a deterministic constant
$\Phi_*$ almost surely. Let
\[\mathcal Z:=\{(x,\eta)\in\TM:\norm\eta\leq(M_u+\ell_R)\bar\alpha\}, \quad  
\Omega_*:=\{z\in\mathcal Z:\Phi_0(z)=\Phi_*,\ 0\in\partial\Phi_0(z)\}.\]
Assume that $\Omega_*\neq\emptyset$ and that $\Phi_0$ satisfies the
KL property
at every point of $\Omega_*$. Then
\begin{equation}\label{eq:finite-length-irpmvr}
 \sum_{k=0}^\infty
 \left(\norm{\wh\eta^k}+\norm{\eta^k}+\norm{e^k}
 +\sqrt{\mathcal E_{k+1}}\right)<+\infty
 \qquad\text{almost surely}.
\end{equation}
Consequently, $\{x^k:k\in\N\}$ converges almost surely to a critical point of
problem~\eqref{eq:model}.
\end{theorem}

\begin{proof}
As in the proof of Lemma~\ref{lem:subsequence}, set
\(\Psi_e(z,e):=\Phi_0(z)+\chi_e\norm e^2.\) Arguing as in the proof of
Lemma~\ref{lem:subsequence}, we now use
the KL argument in Theorem~\ref{thm:abstract-KL} for
$X^k:=(z^k,e^k)$ and $\Psi:=\Psi_e$. As in the proof in
Theorem~\ref{thm:abstract-KL}, since
$e\mapsto\chi_e\norm e^2$ is analytic, the separable-sum rule for the KL
property implies that $\Psi_e$ is KL on $\Omega_*\times\{0\}$. Moreover,
Lemma~\ref{lem:subsequence} yields $e^k\to0$ and
$\mathcal E_{k+1}\to0$ almost surely, hence
$\Psi_e(z^k,e^k)\to\Phi_*$ almost surely. The compactness condition follows
from the compactness of $\M$, Lemma~\ref{lem:step-size}, \eqref{eq:m-bounds},
and the boundedness of $\grad F$ on $\M$. Thus, Items~(iv)--(vi) in
Theorem~\ref{thm:abstract-KL} hold. Items~(i) and (ii) in
Theorem~\ref{thm:abstract-KL} have been verified in the proof of
Lemma~\ref{lem:subsequence} for $\Psi=\Psi_e$ and $X^k=(z^k,e^k)$, and
Item~(iii) is obvious. Thus
Theorem~\ref{thm:abstract-KL} yields \eqref{eq:finite-length-irpmvr}. Finally,
Lemma~\ref{lem:retraction} implies
\[\sum_k\norm{x^{k+1}-x^k}\leq C_R\sum_k\norm{\eta^k}<+\infty\] and hence
$x^k$ converges almost surely to some $x^*\in\M$. Lemma~\ref{lem:subsequence}
implies that $x^*$ is a critical point of problem~\eqref{eq:model}. This
completes the proof.
\end{proof}

For $X^k:=(z^k,e^k)$, the memory terms used in the proof of
Lemma~\ref{lem:subsequence}, Lemma~\ref{lem:retraction}, and the triangle
inequality lead to
\[\norm{X^{k+1}-X^k}\leq C_R\norm{\eta^k}
+\norm{\wh\eta^{k+1}}+\norm{\wh\eta^k}
+\norm{e^{k+1}}+\norm{e^k}\leq C D_{k+1}\] for some $C>0$.
Hence the successive-error condition in Theorem~\ref{thm:abstract-rates}
holds. That theorem and Theorem~\ref{thm:sequential} yield the following
consequence.

\begin{corollary}[KL rates for Algorithm~\ref{alg:iRPMVR} in the deterministic
case]\label{cor:irpmvr-rates}
Suppose that $n=1$, $\gamma_k\equiv1$, and that the assumptions and parameter
conditions in Theorem~\ref{thm:sequential} hold. Let $x^*$ be the critical
point obtained therein, set $z^k:=(x^k,\wh\eta^k)$ and
$z^*:=(x^*,0)$, and assume that $\Phi_0$ defined by
\eqref{eq:shifted-objective} has KL exponent
$\theta\in[0,1)$ at $z^*$.
If $\kappa_2=0$ and $\theta\in[0,1/2]$, then
\[\abs{\Phi_0(z^k)-\Phi_0(z^*)}=O(\rho^k), \quad
\norm{\wh\eta^k}+\norm{\eta^k}=O(\rho^k), \quad \norm{x^k-x^*}=O(\rho^k)\]
for some $\rho\in(0,1)$. In all remaining cases covered by
Theorem~\ref{thm:abstract-rates}(ii), the corresponding polynomial rates in
that theorem hold with $\sigma_k=O(\kappa_2k^{-r})$.
\end{corollary}
\section{Complexity analysis}
In this section, we establish the outer-iteration complexity and the overall
inner complexity of iRPMVR. We first introduce a model-based criticality
measure whose vanishing characterizes the critical points of
problem~\eqref{eq:model}.

\begin{definition}[Model-based criticality measure]\label{def:Theta}
For $x\in\M$ and $\eta\in T_x\M$, define
\[\Theta(x,\eta):= \sqrt{\dist^2\bigl(0,\nabla F(x+\eta)+\partial R(x+\eta)+N_x\M\bigr) +\norm\eta^2}.\]
A point $x\in\M$ is called an $\eps$-critical point if
$\Theta(x,\eta)\leq\eps$ for some $\eta\in T_x\M$.
\end{definition}

For every iterate $x^k\in\M$, let $\wh\eta^k\in T_{x^k}\M$ be the exact
minimizer of $F_k$ over $T_{x^k}\M$; see \eqref{eq:exact-step}. We write
$\Theta_k:=\Theta(x^k,\wh\eta^k)$. The following estimate is the bridge between
the criticality measure $\Theta_k$ and the quantities appearing in the descent
estimate of Theorem~\ref{thm:descent}.

\begin{proposition}[Criticality estimates]\label{prop:criticality}
Suppose Assumptions~\ref{ass:geometry} and \ref{ass:estimator} hold. Let
$\{(x^k,\eta^k):k\in\N\}$ be generated by Algorithm~\ref{alg:iRPMVR}, and let
$\wh\eta^k$ be defined by \eqref{eq:exact-step}. Then the following estimates
hold with $C_\Theta:=1+2(\ell_f+\alpha^{-1})^2$.
\begin{enumerate}
\item[(i)] For every $k\geq0$,
$\Theta_k^2\leq C_\Theta\norm{\wh\eta^k}^2+2\norm{e^k}^2.$
\item[(ii)] For every $k\geq0$,
$\E[\Theta_k^2]\leq C_\Theta\E[\norm{\wh\eta^k}^2] +2c_e\E[\mathcal E_k-\mathcal E_{k+1}] +2\nu\E[\norm{\eta^k}^2].$
\end{enumerate}
\end{proposition}

\begin{proof}
Fix $k\geq0$ and for simplicity, write \[x:=x^k, \quad  \wh\eta:=\wh\eta^k, \quad m:=m^k, \quad e:=e^k=m^k-\grad F(x^k).\] Since $\wh\eta$ is the exact
minimizer of \eqref{eq:exact-step}, there exist
$\xi\in\partial R(x+\wh\eta)$ and $\zeta\in N_x\M$ such that
$m+\xi+\alpha^{-1}\wh\eta+\zeta=0$. Therefore, we get
\[\nabla F(x+\wh\eta)+\xi+\zeta+\grad F(x)-\nabla F(x) =\nabla F(x+\wh\eta)-\nabla F(x)-e-\alpha^{-1}\wh\eta.\]
Using the Lipschitz continuity of $\nabla F$ and
$\grad F(x)-\nabla F(x)\in N_x\M$, we obtain Item~(i). Taking conditional
expectation in Item~(i) and using Assumption~\ref{ass:estimator} together with
the tower property of conditional expectation, we obtain Item~(ii).
\end{proof}

To accommodate history-dependent inner accuracies while retaining a tractable
complexity analysis, we impose the following tolerance rule in this section.

\begin{assumption}[History-dependent inner tolerance]\label{ass:complexity-tol}
Set $\eta^{-1}:=0$. There exist a constant $\kappa\geq0$ and a deterministic
sequence $\{\varsigma_k\geq0:k\in\N\}\subseteq[0,\bar\Delta]$ such that
\begin{equation*}
 \Delta_k=\min\left\{\bar\Delta,
 \max\{\kappa\norm{\eta^{k-1}}^2,\varsigma_k\}\right\},
 \qquad k\in\N.
\end{equation*}
\end{assumption}

The history-dependent term \(\kappa\|\eta^{k-1}\|^2\) allows the inexactness
to be absorbed into the outer descent estimate, whereas the deterministic
floor \(\varsigma_k\) provides a prescribed lower bound on the tolerance,
preventing it from becoming arbitrarily small and thereby enabling an a priori
bound on the total inner work.
Since $\eta^{k-1}$ is $\calF_k$-measurable,
Assumption~\ref{ass:complexity-tol} ensures that $\Delta_k$ is selected from
the available history before the $k$-th inner solve. Moreover, we have
\(\varsigma_k\leq\Delta_k\leq \kappa\norm{\eta^{k-1}}^2+\varsigma_k.\)

\subsection{Outer and IFO complexities for expected stationarity}
We next derive the outer-iteration complexity of Algorithm~\ref{alg:iRPMVR} for
expected stationarity and the corresponding incremental first-order oracle
(IFO) complexities, where an IFO call returns first-order information for one
component function $f_j$. The proof combines the conditional expected descent
estimate in Theorem~\ref{thm:descent} with the criticality estimates in
Proposition~\ref{prop:criticality}.

\begin{theorem}[Outer complexity for expected stationarity]
\label{thm:outer-complexity}
Suppose Assumptions~\ref{ass:geometry}, \ref{ass:estimator}, and
\ref{ass:complexity-tol} hold. Let $C_\alpha$ be from
Lemma~\ref{lem:step-size} and $C_0$ from Lemma~\ref{lem:B-bounds}, and assume
that
\[0<\bar\alpha<\min\{C_\alpha,
1/(L_F+2L_R+2\sqrt{\nu/\gamma}+4\ell_RC_0\kappa/\gamma)\}.\] Let
$\sigma_0:=1/(2\alpha)-L_F/2-L_R$, $\tau:=\sqrt{\nu/\gamma}$, and define
\(\mathcal L_k:=\Phi(x^k)+c_e\mathcal E_k/(2\tau).\) Set
\(a_\eta:=\gamma(\sigma_0-\tau/2)-\nu/(2\tau)-2\ell_RC_0\kappa\), and \(a_{\wh\eta}:=\gamma/(2\alpha).\) Then $a_\eta>0$. Assume further that
$\sum_{k=0}^\infty\varsigma_k<+\infty$, and define
\[\widetilde C:=\Phi(x^0)-\inf_{x\in\M}\Phi(x)
+2\ell_RC_0\sum_{k=0}^\infty\varsigma_k, \quad 
C_{\rm comp}:=\max\{C_\Theta/a_{\wh\eta},2\nu/a_\eta\},\] where $C_\Theta$
is from Proposition~\ref{prop:criticality}. Then the following statements hold.
\begin{enumerate}
\item[(i)] For every $K\geq1$,
$\sum_{k=0}^{K-1}\E[\Theta_k^2]\leq C_{\rm comp}\widetilde C
+c_e(2+C_{\rm comp}/(2\tau))\E[\mathcal E_0]$.
\item[(ii)] Algorithm~\ref{alg:iRPMVR} produces an iterate satisfying
$\min_{0\leq k\leq K-1}\E[\Theta_k^2]\leq\eps^2$ within
$O\bigl([C_{\rm comp}+c_e(1+C_{\rm comp}/\tau)\E[\mathcal E_0]]\eps^{-2}\bigr)$
outer iterations.
\end{enumerate}
\end{theorem}

\begin{proof}
The condition on $\bar\alpha$ implies $a_\eta>0$. By
Theorem~\ref{thm:descent}, taking full expectation and using
$\gamma_k\geq\gamma$, we obtain
\[\E[\mathcal L_{k+1}] \leq\E[\mathcal L_k]-(a_\eta+2\ell_RC_0\kappa) \E[\norm{\eta^k}^2]-a_{\wh\eta}\E[\norm{\wh\eta^k}^2] +2\ell_RC_0\E[\Delta_k].\]
By Assumption~\ref{ass:complexity-tol},
$\E[\Delta_k]\leq\kappa\E[\norm{\eta^{k-1}}^2]+\varsigma_k$.
Summing the preceding inequality from $k=0$ to $K-1$, using $\eta^{-1}=0$,
and shifting the index yield
\[a_\eta\sum_{k=0}^{K-1}\E[\norm{\eta^k}^2] +a_{\wh\eta}\sum_{k=0}^{K-1}\E[\norm{\wh\eta^k}^2] \leq\E[\mathcal L_0]-\inf_{x\in\M}\Phi(x) +2\ell_RC_0\sum_{k=0}^{K-1}\varsigma_k.\]
Moreover,
$\sum_{k=0}^{K-1}\E[\mathcal E_k-\mathcal E_{k+1}]
\leq\E[\mathcal E_0]$. Combining these inequalities with
Proposition~\ref{prop:criticality}(ii) and
\(\E[\mathcal L_0]=\Phi(x^0)+c_e\E[\mathcal E_0]/(2\tau)\) implies Item~(i).
Item~(ii) follows from
\(K\min_{0\leq k<K}\E[\Theta_k^2]\leq
\sum_{k=0}^{K-1}\E[\Theta_k^2].\)
\end{proof}

Since Algorithm~1 performs one retraction at each outer iteration, we derive the same
bound for retraction complexity. The following result combines Theorem~\ref{thm:outer-complexity} and
Proposition~\ref{prop:VR-parameters} with the per-iteration IFO costs
$O(b+n/q_{\rm ep})$ for SVRG and SARAH/SPIDER and $O(b)$, after $O(n)$
initialization, for SAGA and SAG. For SVRG, SARAH/SPIDER, and SAGA, the
resulting orders match those of their standard Euclidean nonconvex counterparts
\cite{Reddi2016SVRG,Reddi2016Proximal,Pham2020,Fang2018SPIDER}.

\begin{theorem}[Concrete VR outer and IFO complexity]
\label{thm:VR-complexity}
Suppose Assumptions~\ref{ass:geometry} and \ref{ass:complexity-tol} hold, and
initialize the VR state as in Proposition~\ref{prop:VR-parameters}. Assume that
$\sum_{k=0}^\infty\varsigma_k\leq C_{\rm sum}$, where $C_{\rm sum}$ and
$\kappa$ are independent of $n$ and $\eps$. For the four estimators, choose
\[
\begin{array}{c|c|c|c}
\textnormal{Estimator} & \textnormal{Parameters}
 & \textnormal{Outer comp.} & \textnormal{IFO complexity}\\
\hline
\textnormal{SVRG}
 & q_{\rm ep}=\lceil n^{1/3}\rceil,\ b=\lceil n^{2/3}\rceil
 & O(\eps^{-2}) & O(n+n^{2/3}\eps^{-2})\\
\textnormal{SARAH/SPIDER}
 & q_{\rm ep}=\lceil n^{1/2}\rceil,\ b=\lceil n^{1/2}\rceil
 & O(\eps^{-2}) & O(n+n^{1/2}\eps^{-2})\\
\textnormal{SAGA}
 & b=\lceil n^{2/3}\rceil
 & O(\eps^{-2}) & O(n+n^{2/3}\eps^{-2})\\
\textnormal{SAG}
 & b=\lceil\kappa_b n\rceil
 & O(\eps^{-2}) & O(n+n\eps^{-2})
\end{array}
\]
where $q_{\rm ep}$ is used only for SVRG and SARAH/SPIDER,
$\kappa_b\in(0,1)$ is a fixed constant independent of $n$ and $\eps$. Then
there exists $\alpha_0>0$, independent of $n$, such that, for every fixed
$0<\alpha\leq\bar\alpha\leq\alpha_0$, Algorithm~\ref{alg:iRPMVR} produces an
iterate satisfying $\E[\Theta_k^2]\leq\eps^2$ within the stated outer and IFO
complexities.
\end{theorem}
\begin{proof}
	For the parameter choices in the table, Proposition~\ref{prop:VR-parameters}
	gives $\E[\mathcal E_0]=0$ and
	\[
	\frac{q_{\rm ep}^2}{b}=O(1),\qquad
	\frac{q_{\rm ep}}{b}=O(1),\qquad
	\frac{n^2}{b^3}=O(1),\qquad
	\frac{n^2}{b^2}=O(1)
	\]
	for SVRG, SARAH/SPIDER, SAGA, and SAG, respectively. Moreover,
	$c_e=O(1)$ in all four cases. Hence, there exists a constant
	$\bar\nu>0$, independent of $n$ and $\eps$, such that $\nu\leq\bar\nu$.
	Choose
	\[
	0<\alpha_0<
	\min\left\{C_\alpha,
	\frac{1}{L_F+2L_R+2\sqrt{\bar\nu/\gamma}
		+4\ell_RC_0\kappa/\gamma}\right\}.
	\]
	Then, for every fixed $0<\alpha\leq\bar\alpha\leq\alpha_0$, the stepsize
	condition in Theorem~\ref{thm:outer-complexity} holds uniformly in $n$.
	Since $\sum_{k=0}^\infty\varsigma_k\leq C_{\rm sum}$ and
	$\E[\mathcal E_0]=0$, that theorem yields an outer complexity of
	$K=O(\eps^{-2})$, with the hidden constant independent of $n$ and $\eps$.	
	For SVRG and SARAH/SPIDER, the total IFO cost is
	\(
	O\bigl(n+K(b+n/q_{\rm ep})\bigr),
	\)
	whereas for SAGA and SAG it is $O(n+Kb)$. Substituting the parameter choices
	in the table and using $K=O(\eps^{-2})$ gives the stated IFO complexities.
\end{proof}
\subsection{Inner solvers and overall complexity}
We now briefly discuss several implementations of the inner routine. Recall
from \eqref{eq:Gk} that, at the $k$-th outer iteration, the dual problem is
$\min_{\lambda\in\R^p}G_k(\lambda)$, where
$G_k(\lambda):=F_k^*(-B_{x^k}^{\trans}\lambda)$, and
Algorithm~\ref{alg:iRPMVR} requires a
dual point $\lambda^k$ satisfying
$\norm{\nabla G_k(\lambda^k)}\leq\Delta_k$. As discussed in Section~3,
we obtain from Lemma~\ref{lem:B-bounds} that $G_k$ is convex and has an
$L_D$-Lipschitz continuous gradient, where
$L_D=\alpha\overline C_B^2$ is independent of $k$ and $n$. For each
$k\in\N$, the dual problem admits an optimal solution $\lambda_k^*$, and these
solutions can be chosen uniformly bounded. Indeed, the KKT conditions for
\eqref{eq:exact-step} yield a dual optimal solution $\lambda_k^*\in\R^p$ and
$\xi^k\in\partial R(x^k+\wh\eta^k)$ such that
$B_{x^k}^{\trans}\lambda_k^*=-m^k-\alpha^{-1}\wh\eta^k-\xi^k.$
Hence, Lemma~\ref{lem:B-bounds}, \eqref{eq:m-bounds},
Lemma~\ref{lem:step-size}, and $\norm{\xi^k}\leq\ell_R$ yield
\[\norm{\lambda_k^*}\leq\frac{2(M_u+\ell_R)}{\underline C_B}, \quad k\in\N.\]
We consider the following three implementations of the inner routine.
\paragraph{\normalfont\bfseries Algorithm 1-FISTA}
The first implementation applies Nesterov's accelerated gradient method, or
equivalently FISTA with a zero nonsmooth term, to
$\min_{\lambda\in\R^p}G_k(\lambda)$; see
\cite[Theorem~2.2]{Nesterov2018} and \cite[Theorem~4.4]{BeckTeboulle2009}.
Let $\lambda_k^*$ be an optimal solution and initialize the inner iteration at
zero. The standard objective-gap estimate and
$\norm{\nabla G_k(\lambda)}^2 \leq2L_D(G_k(\lambda)-\inf G_k)$
yield
\[\norm{\nabla G_k(\lambda^{k,j})}\leq2L_D\norm{\lambda_k^*}/(j+1).\]
Hence, by the uniform boundedness of $\{\lambda_k^*:k\in\N\}$,
$K_{\rm FISTA}(\Delta_k)=O(\Delta_k^{-1})$ inner iterations suffice to obtain
$\norm{\nabla G_k(\lambda^k)}\leq\Delta_k$.
\paragraph{\normalfont\bfseries Algorithm 1-NFG}
The second implementation follows Nesterov's fast-gradient (NFG) in
\cite[Section~5.2]{Jiang2025}. In this approach, NFG is applied not to the
original dual problem directly, but to a regularized dual problem
\[\min_{\lambda\in\R^p}\ G_k(\lambda)+\frac{\delta_k}{2}\norm\lambda^2,\]
where $\delta_k>0$ is chosen proportional to the target residual $\Delta_k$.
By the NFG estimate in \cite[Lemma~5.4]{Jiang2025}, which is based on the
accelerated-gradient theory in
\cite[Section~2.2.2 and Theorem~2.2.7]{Nesterov2018}, one obtains
$K_{\rm NFG}(\Delta_k) =O\left(\Delta_k^{-1/2}\log(1+\Delta_k^{-1})\right)$
inner iterations for satisfying the original dual residual condition.
\paragraph{\normalfont\bfseries Algorithm 1-AR}
The third implementation follows the accumulative regularization (AR) strategy used in \cite[Section~5.2]{Jiang2025}. This method solves a sequence of regularized dual problems with moving proximal centers and increasing regularization parameters. The corresponding complexity estimate follows from
\cite[Theorem~2.1]{Lan2026} together with the accelerated-gradient/FISTA estimate for the regularized subproblems; see also \cite[Theorem~5.7]{Jiang2025} and \cite[Theorem~10.34]{Beck2017}. In the present notation, this implies $K_{\rm AR}(\Delta_k)=O(\Delta_k^{-1/2})$ inner iterations for computing a dual point satisfying $\norm{\nabla G_k(\lambda^k)}\leq\Delta_k$.

We next combine the above inner-solver bounds with
Theorem~\ref{thm:outer-complexity} by choosing the deterministic tolerance
floor $\varsigma_k$ according to the target accuracy.

\begin{theorem}[Overall inner complexity]\label{thm:inner-complexity}
Suppose Assumptions~\ref{ass:geometry}, \ref{ass:estimator} and \ref{ass:complexity-tol} hold. Let
$\kappa\geq0$ and let $\bar\alpha$ satisfy the stepsize condition in
Theorem~\ref{thm:outer-complexity}, and define $\tau$, $a_\eta$, $a_{\wh\eta}$, and
$C_{\rm comp}$ as therein. Let $C_0$ be as in Lemma~\ref{lem:B-bounds}, and set
\(C_{\rm tol}:=\max\{1,4C_{\rm comp}\ell_RC_0\}\) and  \[A_0:=C_{\rm comp}\left(\Phi(x^0)-\inf_{x\in\M}\Phi(x)\right) +c_e\left(2+\frac{C_{\rm comp}}{2\tau}\right)\E[\mathcal E_0].\]
Let $\eps>0$ be sufficiently small so that
$\eps^2/C_{\rm tol}\leq\bar\Delta$, and set
$K_\eps:=\max\{1,\lceil2A_0\eps^{-2}\rceil\}$. Choose $\Delta_k$ according to
Assumption~\ref{ass:complexity-tol}, with
$\varsigma_k:=\eps^2/C_{\rm tol}$ for $k=0,\ldots,K_\eps-1$ and
$\varsigma_k:=0$ for $k\geq K_\eps$. Then
$\min_{0\leq k<K_\eps}\E[\Theta_k^2]\leq\eps^2$. The total numbers of inner iterations required by Algorithm 1-FISTA, Algorithm 1-NFG, and Algorithm 1-AR are $O(\eps^{-4})$, $O\bigl(\eps^{-3}\log(1+\eps^{-2})\bigr)$, and $O(\eps^{-3})$, respectively.
\end{theorem}

By Danskin's theorem \cite[Proposition~B.22]{Bertsekas2016}, we have
\[
\nabla G_k(\lambda)
=
-B_{x^k}
\left[
\operatorname{prox}_{\alpha R}
\left(
x^k-\alpha\bigl(m^k+B_{x^k}^{\top}\lambda\bigr)
\right)-x^k
\right].
\]
Since each inner iteration of the above first-order dual solvers uses
\(O(1)\) evaluations of \(\nabla G_k\), and each evaluation of
\(\nabla G_k\) requires one evaluation of
\(\operatorname{prox}_{\alpha R}\), the above inner-iteration bounds
also give the same-order bounds on the total number of proximal-operator
evaluations for \(R\).
\section{Verification for concrete variance-reduced estimators}
In this section, we verify Assumption~\ref{ass:estimator} for projection-based
variants of SVRG \cite{Reddi2016SVRG}, SARAH/SPIDER
\cite{Pham2020,Fang2018SPIDER}, SAGA \cite{Reddi2016Proximal}, and SAG
\cite{Schmidt2017}. We first describe the base estimator $u^k$ used in
Algorithm~\ref{alg:iRPMVR}. Let $b\in\{1,\ldots,n\}$ denote the mini-batch
size. For the epoch-based estimators, let $q_{\rm ep}\in\N$ denote the epoch
length. All mini-batches are sampled uniformly from $\{1,\ldots,n\}$.

\paragraph{\normalfont\bfseries SVRG estimator}
At the beginning of the $s$th epoch, set
$\wt x^s:=x^{sq_{\rm ep}}$ and compute $\grad F(\wt x^s)$. For
$k=sq_{\rm ep},\ldots,(s+1)q_{\rm ep}-1$, define
\[u^k:=P_{T_{x^k}\M}\grad F(\wt x^s) +\frac1b\sum_{j\in\mathcal B_k} \left(\grad f_j(x^k)-P_{T_{x^k}\M}\grad f_j(\wt x^s)\right).\]

\paragraph{\normalfont\bfseries SARAH/SPIDER estimator}
At the beginning of the $s$th epoch, set
$u^{sq_{\rm ep}}:=\grad F(x^{sq_{\rm ep}})$. For
$k=sq_{\rm ep}+1,\ldots,(s+1)q_{\rm ep}-1$, define recursively
\[u^k:=P_{T_{x^k}\M}u^{k-1} +\frac1b\sum_{j\in\mathcal B_k} \left(\grad f_j(x^k)-P_{T_{x^k}\M}\grad f_j(x^{k-1})\right).\]
This estimator uses only gradient differences between consecutive iterates,
and the full gradient is recomputed every $q_{\rm ep}$ iterations.

\paragraph{\normalfont\bfseries SAGA estimator}
For each $j=1,\ldots,n$, let $z_j^k\in\M$ be the stored point and set
$y_j^k:=\nabla f_j(z_j^k)\in\R^d$ and
$\bar y^k:=\frac{1}{n}\sum_{j=1}^ny_j^k$. The SAGA estimator is defined by
\[v^k:=\frac1b\sum_{j\in\mathcal B_k} \left(\nabla f_j(x^k)-y_j^k\right)+\bar y^k,\quad  u^k:=P_{T_{x^k}\M}v^k.\]
After $u^k$ is formed, the table entries indexed by $\mathcal B_k$ are
refreshed according to
$z_j^{k+1}:=x^k$, $y_j^{k+1}:=\nabla f_j(x^k)$ and $j\in\mathcal B_k$, while
the remaining entries are left unchanged. The table is initialized by
$z_j^0=x^0$ and $y_j^0=\nabla f_j(x^0)$ for $j=1,\ldots,n$.

\paragraph{\normalfont\bfseries SAG estimator}
The SAG estimator uses the same ambient gradient table as SAGA, but replaces
the unbiased SAGA correction by an averaged correction:
\[v^k:=\bar y^k+\frac1n\sum_{j\in\mathcal B_k} \left(\nabla f_j(x^k)-y_j^k\right), \quad u^k:=P_{T_{x^k}\M}v^k.\]
The table is refreshed in the same way as in the SAGA estimator. Unlike SAGA,
this estimator is generally biased, but the bias is controlled by the same
memory error sequence.

We first isolate the effect of the momentum projection, so that it suffices to
verify the base estimator $u^k$ for each variance-reduction mechanism.

\begin{lemma}[Momentum lifting]\label{lem:momentum}
Suppose Assumption~\ref{ass:geometry} holds. Let
$\delta^k:=u^k-\grad F(x^k)$ and
$M^k:=\bar m^k-\grad F(x^k)$. Assume that there exist constants
$c_u>0$, $\mu_u>0$, $\nu_u>0$, and a nonnegative adapted sequence
$\{V_k:k\in\N\}$ such that
\begin{equation}\label{eq:base-estimator}
 \E_k[\norm{\delta^k}^2]+\mu_uV_k
 \leq c_u(V_k-\E_k[V_{k+1}])+\nu_u\E_k[\norm{\eta^k}^2].
\end{equation}
Then Assumption~\ref{ass:estimator} holds with
$\mathcal E_k:=V_k+\lambda\norm{M^k}^2$ for some $\lambda>0$. Moreover,
\begin{equation}\label{eq:momentum-orders}
 c_e=C_uc_u,\qquad
 \mu=C_u'\min\{\mu_u,c_u\},\qquad
 \nu=C_u''(\nu_u+1),
\end{equation}
where $C_u,C_u',C_u''>0$ may depend on $\bar\chi$ and $L_T$, but not on
$n,b,q_{\rm ep},K$, or $\eps$.
\end{lemma}

\begin{proof}
Set $\wh\delta^k:=\wh u^k-\grad F(x^k)$. By the nonexpansiveness of the clipping
projection, $\norm{\wh\delta^k}\leq\norm{\delta^k}$, and hence
$\norm{e^k}^2\leq\bar\chi\norm{M^k}^2+\norm{\delta^k}^2.$
Choose $\beta>0$ such that $\rho_M:=(1+\beta)\bar\chi<1$. By the momentum
update, Lemma~\ref{lem:projection}, and Young's inequality,
\[\E_k[\norm{M^{k+1}}^2] \leq\rho_M\norm{M^k}^2+A_M\E_k[\norm{\delta^k}^2] +B_M\E_k[\norm{\eta^k}^2],\]
where $A_M:=1+\beta$ and $B_M:=(1+\beta^{-1})L_T^2$.
Set
$\lambda:=(2A_Mc_u)^{-1}$,
$c_e:=Cc_u$, and $\mu:=c_\mu\min\{\mu_u,c_u\}$, where
$c_\mu:=(1-\rho_M)/2$ and
$C\geq\max\{2,1/2+2A_M\bar\chi/(1-\rho_M)\}.$
Combining the preceding estimate with \eqref{eq:base-estimator}, we get
\[\E_k[\norm{e^k}^2]+\mu\mathcal E_k \leq
c_e(\mathcal E_k-\E_k[\mathcal E_{k+1}])
+\nu\E_k[\norm{\eta^k}^2],\] with $\nu:=C(\nu_u+B_M/(2A_M))$. This proves
\eqref{eq:momentum-orders}.
\end{proof}

We now summarize the estimator parameters for the four variance-reduction
mechanisms introduced above.

\begin{proposition}[Estimator parameters for momentum VR schemes]
\label{prop:VR-parameters}
Suppose Assumption~\ref{ass:geometry} holds and
$\bar\alpha\leq C_\alpha$, with $C_\alpha$ from
Lemma~\ref{lem:step-size}. Assume that the variance-reduction state is
initialized at $x^0$ so that $u^0=\grad F(x^0)$, and that
$\bar m^0=\grad F(x^0)$. Then the SVRG, SARAH/SPIDER, SAGA, and SAG estimators
defined above, combined with the momentum update in Algorithm~\ref{alg:iRPMVR},
satisfy Assumption~\ref{ass:estimator} with the following parameter orders:
\[
\begin{array}{c|c|c|c|c}
\text{Estimator} & c_e & \mu & \nu & \E[\mathcal E_0]\\
\hline
\text{SVRG} & O(1) & O(q_{\rm ep}^{-1}) & O(q_{\rm ep}^2/b+1) & 0\\
\text{SARAH/SPIDER} & O(1) & O(q_{\rm ep}^{-1}) & O(q_{\rm ep}/b+1) & 0\\
\text{SAGA} & O(n/b^2) & O(b^{-1}) & O(n^2/b^3+1) & 0\\
\text{SAG} & O(n/b) & O(b^{-1}) & O(n^2/b^2+1) & 0
\end{array}
\]
\end{proposition}

\begin{proof}
By Lemma~\ref{lem:momentum}, it suffices to verify
\eqref{eq:base-estimator} for each base estimator.

For SVRG, when $n=1$, the variance term is zero and the estimate is
trivial.
Thus we only need to consider $n\geq2$ in the following variance estimate.
Let $s=s(k):=\lfloor k/q_{\rm ep}\rfloor$, so that
$\wt x^s=x^{sq_{\rm ep}}$, and write $P_k:=P_{T_{x^k}\M}$. Define
\[d_j^k:=\grad f_j(x^k)-P_k\grad f_j(\wt x^s), \quad
\bar d^k:=\frac1n\sum_{j=1}^nd_j^k.\]
Then $\delta^k=b^{-1}\sum_{j\in\mathcal B_k}(d_j^k-\bar d^k)$.
By the sampling-without-replacement variance identity
\cite[Lemma~1]{Mishchenko2020}, applied conditionally on $\calF_k$, we have
\[\E_k[\delta^k]=0, \quad 
\E_k[\norm{\delta^k}^2]
=\frac{n-b}{b(n-1)}\frac1n\sum_{j=1}^n\norm{d_j^k-\bar d^k}^2.\] For the sampling-with-replacement case, we know \[\E_k[\delta^k]=0, \quad
\E_k[\norm{\delta^k}^2]=\frac1{nb}\sum_{j=1}^n\norm{d_j^k-\bar d^k}^2.\]
These imply that
$\E_k[\norm{\delta^k}^2]\leq\frac1{bn}\sum_{j=1}^n\norm{d_j^k-\bar d^k}^2$.
By Lemma~\ref{lem:projection}, $\norm{d_j^k}\leq L_T\sum_{\ell=sq_{\rm ep}}^{k-1}
\gamma_\ell\norm{\eta^\ell}$ for every $j=1,\ldots,n$.
Therefore, by Cauchy's inequality and $k-sq_{\rm ep}\leq q_{\rm ep}$,
\begin{equation}\label{eq:svrg-variance}
 \E_k[\norm{\delta^k}^2]
 \leq\frac{L_T^2}{b}
 \left(\sum_{\ell=sq_{\rm ep}}^{k-1}
 \gamma_\ell\norm{\eta^\ell}\right)^2
 \leq\frac{L_T^2q_{\rm ep}}b
 \sum_{\ell=sq_{\rm ep}}^{k-1}\gamma_\ell^2\norm{\eta^\ell}^2.
\end{equation}
Set
\[H_k:=\sum_{\ell=sq_{\rm ep}}^{k-1}\gamma_\ell^2\norm{\eta^\ell}^2, \quad
a_{\rm svrg}:=L_T^2q_{\rm ep}/b, \quad
\mathcal V_k:=a_{\rm svrg}(q_{\rm ep}-k+sq_{\rm ep})H_k.\] Then $\mathcal V_k$ is
nonnegative and $\calF_k$-measurable. Moreover,
\eqref{eq:svrg-variance} yields
$\E_k[\norm{\delta^k}^2]\leq a_{\rm svrg}H_k$. We next estimate the one-step
decrease of $\mathcal V_k$. If
$k-sq_{\rm ep}\leq q_{\rm ep}-2$, then
\[\E_k[\mathcal V_{k+1}]=a_{\rm svrg}(q_{\rm ep}-k-1+sq_{\rm ep})
(H_k+\gamma_k^2\E_k[\norm{\eta^k}^2]).\] If
$k-sq_{\rm ep}=q_{\rm ep}-1$, then $k+1$ is the beginning of the next SVRG
epoch, and hence $\mathcal V_{k+1}=0$. Consequently, in both cases,
\[\mathcal V_k-\E_k[\mathcal V_{k+1}] \geq a_{\rm svrg}H_k-a_{\rm svrg}(q_{\rm ep}-1) \E_k[\norm{\eta^k}^2],\]
where we used $\gamma_k\leq1$. Since
$0\leq k-sq_{\rm ep}\leq q_{\rm ep}$, it also holds that
$q_{\rm ep}^{-1}\mathcal V_k\leq a_{\rm svrg}H_k$. Combining the preceding estimates
yields \[\E_k[\norm{\delta^k}^2]+q_{\rm ep}^{-1}\mathcal V_k\leq2a_{\rm svrg}H_k
\leq2(\mathcal V_k-\E_k[\mathcal V_{k+1}])
+2a_{\rm svrg}(q_{\rm ep}-1)\E_k[\norm{\eta^k}^2].\] Thus \eqref{eq:base-estimator} holds for the SVRG estimator with
$c_u=2$, $\mu_u=q_{\rm ep}^{-1}$ and
$\nu_u=\frac{2L_T^2q_{\rm ep}(q_{\rm ep}-1)}{b}$.

For SARAH/SPIDER, let $s=s(k):=\lfloor k/q_{\rm ep}\rfloor$ and
$r_k:=k-sq_{\rm ep}$. If $r_k=0$, then $u^k=\grad F(x^k)$, and hence
$\delta^k=0$. Suppose $r_k\geq1$ and write $P_k:=P_{T_{x^k}\M}$. Define
\[d_j^k:=\grad f_j(x^k)-P_k\grad f_j(x^{k-1}), \quad
\bar d^k:=n^{-1}\sum_{j=1}^nd_j^k.\] Then
$\delta^k=P_k\delta^{k-1}+\frac1b\sum_{j\in\mathcal B_k}(d_j^k-\bar d^k).$
The second term has conditional mean zero. Hence, by the nonexpansiveness of
$P_k$, the sampling variance estimate used above, and
Lemma~\ref{lem:projection},
\[\E_k[\norm{\delta^k}^2]\leq\norm{\delta^{k-1}}^2
+\frac{L_T^2}{b}\gamma_{k-1}^2\norm{\eta^{k-1}}^2.\]
Set $a_{\rm sp}:=L_T^2/b$. If $r_k=0$, set $H_k:=0$, while if $r_k\geq1$, set
$H_k:=\norm{\delta^{k-1}}^2+a_{\rm sp}\gamma_{k-1}^2\norm{\eta^{k-1}}^2$, and
$\mathcal V_k:=(q_{\rm ep}-r_k+1)H_k$. Then $\mathcal V_k$ is nonnegative and $\calF_k$-measurable. Arguing as in the
preceding proof for SVRG,
\[\E_k[\mathcal V_{k+1}]\leq(q_{\rm ep}-r_k)
(H_k+a_{\rm sp}\E_k[\norm{\eta^k}^2]).\]
Thus, \(\mathcal V_k-\E_k[\mathcal V_{k+1}]
\geq H_k-a_{\rm sp}q_{\rm ep}\E_k[\norm{\eta^k}^2].\)
Therefore, applying
$\E_k[\norm{\delta^k}^2]+q_{\rm ep}^{-1}\mathcal V_k\leq2H_k$,
\eqref{eq:base-estimator} holds for the SARAH/SPIDER estimator with
$c_u=2$, $\mu_u=q_{\rm ep}^{-1}$, and $\nu_u=2L_T^2q_{\rm ep}/b$.

For SAGA, define the ambient memory discrepancy by
$g_j^k:=\nabla f_j(x^k)-y_j^k$ and
$H_k:=\frac1n\sum_{j=1}^n\norm{g_j^k}^2$.
Since
$\bar g^k:=\frac1n\sum_{j=1}^n g_j^k=\nabla F(x^k)-\bar y^k$,
the definition of the SAGA estimator implies
$\delta^k=P_{T_{x^k}\M}\left[\frac1b\sum_{j\in\mathcal B_k}(g_j^k-\bar g^k)\right].$
By the nonexpansiveness of $P_{T_{x^k}\M}$ and by the same sampling
variance estimate used in the SVRG case,
\[\E_k[\norm{\delta^k}^2]\leq\frac1{bn}\sum_{j=1}^n\norm{g_j^k-\bar g^k}^2
\leq\frac1b H_k.\]
We next estimate the table-refresh recursion. Let $I_j^k$ be the indicator of
the event that the index $j$ is refreshed at iteration $k$, and let
$p_b:=\mathbb P(I_j^k=1)$.
For the without-replacement case, $p_b=b/n$; for the with-replacement
case, \[p_b=1-(1-1/n)^b\geq 1-e^{-b/n}\geq(1-e^{-1})b/n,\] where the last
inequality utilizes the concavity of $1-e^{-x}$ on $[0,1]$. Thus, in both cases,
$p_b\geq c_1b/n$ with $c_1=1-e^{-1}$. Set
$s_j^k:=\nabla f_j(x^{k+1})-\nabla f_j(x^k)$.
By the table update rule, if $j$ is refreshed, then
$y_j^{k+1}=\nabla f_j(x^k)$, while otherwise $y_j^{k+1}=y_j^k$. Hence, pathwise,
$g_j^{k+1}=\nabla f_j(x^{k+1})-y_j^{k+1}=s_j^k+(1-I_j^k)g_j^k$.
Applying Young's inequality with parameter $p_b/2$, we obtain
\[\norm{g_j^{k+1}}^2\leq\left(1+\frac{p_b}{2}\right)(1-I_j^k)\norm{g_j^k}^2
+\left(1+\frac{2}{p_b}\right)\norm{s_j^k}^2.\]
Taking conditional expectation and using $\E_k[1-I_j^k]=1-p_b$, 
\[\left(1+\frac{p_b}{2}\right)(1-p_b)
=1-\frac{p_b}{2}-\frac{p_b^2}{2}\leq1-\frac{p_b}{2},\]
and $\norm{s_j^k}\leq\ell_f C_R\norm{\eta^k}$ which is obtained by the
Lipschitz continuity of the ambient gradients and Lemma~\ref{lem:retraction},
we obtain
\[\E_k[H_{k+1}]\leq\left(1-\frac{p_b}{2}\right)H_k
+\frac{3\ell_f^2C_R^2}{p_b}\E_k[\norm{\eta^k}^2],\] where we used $p_b\leq1$.
Hence, since $p_b\geq c_1b/n$,
\begin{equation}\label{eq:table-recursion}
 H_k-\E_k[H_{k+1}]
 \geq c_2\frac bnH_k-C_{\rm tab}\frac nb\E_k[\norm{\eta^k}^2],
\end{equation}
where $c_2:=c_1/2$ and $C_{\rm tab}:=3\ell_f^2C_R^2/c_1$.
Set $\mathcal V_k:=H_k$. Therefore, applying
$\E_k[\norm{\delta^k}^2]+b^{-1}\mathcal V_k\leq2/bH_k$,
\eqref{eq:base-estimator} holds with
$c_u=2n/(c_2b^2)$, $\mu_u=b^{-1}$, and
$\nu_u=2C_{\rm tab}n^2/(c_2b^3)$.

For SAG, use the same notations $g_j^k,H_k,\bar g^k$ as in the SAGA case. By
the definition of the SAG estimator,
$\delta^k=P_{T_{x^k}\M}\left(\frac1n\sum_{j\in\mathcal B_k}g_j^k-\bar g^k\right)$.
The SAG correction is biased, so we do not use the centered sampling-variance
identity used for SAGA. Instead, by the nonexpansiveness of
$P_{T_{x^k}\M}$, Cauchy's inequality, and $\norm{\bar g^k}^2\leq H_k$, we have
\[\E_k[\norm{\delta^k}^2]\leq2\E_k\left[\norm{\frac1n\sum_{j\in\mathcal B_k}g_j^k}^2\right]
+2\norm{\bar g^k}^2\leq4H_k,\]
where the last inequality is valid for both with- and without-replacement
sampling, because
$\E_k\left[\norm{\frac1n\sum_{j\in\mathcal B_k}g_j^k}^2\right]
\leq b^2/{n^2}H_k$.
The table-refresh recursion \eqref{eq:table-recursion} is exactly the same as in
the SAGA case. Arguing as in the proof for SAGA and setting
$\mathcal V_k:=H_k$, we obtain that \eqref{eq:base-estimator} holds with
$c_u=5n/(c_2b)$, $\mu_u=b^{-1}$, and
$\nu_u=5C_{\rm tab}n^2/(c_2b^2)$,
where $c_2$ and $C_{\rm tab}$ are the same constants as in the SAGA case.

In all four cases, $q_{\rm ep}\geq1$ and $1\leq b\leq n$ imply
$\min\{\mu_u,c_u\}=\mu_u$. Hence Lemma~\ref{lem:momentum} yields the parameter
orders stated in the table.
Finally, the stated initialization $u^0=\bar m^0=\grad F(x^0)$ implies that
the corresponding base memory satisfies $\mathcal V_0=0$, while
$M^0=\bar m^0-\grad F(x^0)=0$. Therefore
$\mathcal E_0=\mathcal V_0+\lambda\norm{M^0}^2=0$. This completes the proof.
\end{proof}

Thus, all four variance-reduction mechanisms satisfy Assumption~2, so the
convergence and complexity results developed above apply with the
corresponding parameter orders.
\section{Conclusion}

We developed a unified analysis of inexact stochastic Riemannian proximal
optimization for finite-sum nonsmooth composite problems over compact embedded
submanifolds. The analysis is organized around a conditional
error-dissipation condition for the gradient-surrogate errors and a predictable
inexactness rule for the tangent-space proximal subproblems. This structure
decouples the main convergence and complexity arguments from the particular
variance-reduction mechanism and inner solver, and covers projection-based
variants of SVRG, SARAH/SPIDER, SAGA, and SAG. Together with a computable
Fenchel-dual residual criterion, it yields subsequential stationarity,
\(O(\epsilon^{-2})\) outer-iteration and retraction complexity, as well as
overall oracle-complexity bounds that account explicitly for the inexact inner
solves.

We also established an abstract KL principle for conditional expected descent
relations with memory and summable tail perturbations. The analysis relies only
on the ordinary pointwise KL property and avoids the expected-KL implication
whose validity can fail, as demonstrated by our counterexample. Applied to
full-step iRPMVR, the principle yields almost-sure finite length and
whole-sequence convergence, together with deterministic KL rates. Possible extensions include expectation-model and online stochastic settings,
as well as replacing global compactness by suitable boundedness and local
uniformity conditions and relaxing the deterministic-limit-value requirement
in the sequential convergence analysis.
\bibliographystyle{siamplain}
\bibliography{references}
\end{document}